\documentclass[11pt]{amsart}
\usepackage[all, knot, color]{xy}
\usepackage{amsmath,amsfonts,amssymb,amsthm,color,dsfont,amsbsy}
\usepackage{marginnote}
\usepackage{changepage}
\usepackage{amscd}

\usepackage{tikz}
\usetikzlibrary{cd, calc, knots}

\usepackage{graphicx}
\def\ba{\pmb{a}}
\def\bsa{\pmb{sa}}

\def\a{{\alpha}}
\def\b{{\beta}}
\def\l{{\lambda}}

\def\w{{\omega}}
\def\g{{\gamma}}
\def\o{{\otimes}}
\def\s{{\sigma}}
\def\t{{\tau}}
\def\bq{{\pmb{p}}}

\def\SL2{{\mbox{SL}_2(\BZ)}}
\def\Mp2{{\mbox{Mp}_2(\BZ)}}

\def\1{{\mathds{1}}}
\def\to{\rightarrow}
\def\qed{{\ \ \ \mbox{$\square$}}}

\def\BC{{\mathbb{C}}}
\def\BN{{\mathbb{N}}}
\def\BR{{\mathbb{R}}}
\def\BZ{{\mathbb{Z}}}
\def\BQ{{\mathbb{Q}}}
\def\BF{{\mathbb{F}}}

\def\CC{\mathcal{C}}
\def\DD{\mathcal{D}}
\def\BB{\mathcal{B}}
\def\UU{\mathcal{U}}
\def\U{\mathcal{U}_\infty}

\def\Gal{{\operatorname{Gal}}}

\def\End{{\operatorname{End}}}

\def\Rep{{\operatorname{Rep}}}

\newcommand\Irr{{\operatorname{Irr}}}

\def\id{{\operatorname{id}}}
\def\lcm{{\operatorname{lcm}}}

\def\Tor{{\operatorname{Tor}}}
\def\Tr{{\operatorname{tr}}}

\def\ord{{\operatorname{ord}\,}}

\def\ob{{\operatorname{ob}}}

\def\Vs{{\operatorname{Vec}}}

\def\inv{^{-1}}

\def\ev{{\operatorname{ev}}}
\def\coev{{\operatorname{coev}}}

\renewcommand{\Xi}{\Psi}

\renewcommand\subset{ \subseteq }

\newtheorem{thm}{Theorem}[section]
\newtheorem{cor}[thm]{Corollary}
\newtheorem{prop}[thm]{Proposition}
\newtheorem{lem}[thm]{Lemma}

\newtheorem*{theorem*}{Theorem}

\theoremstyle{definition}

\newtheorem{defn}[thm]{Definition}
\newtheorem{remark}[thm]{Remark}

\newtheorem{question}[thm]{Question}

\newcommand{\dcox}{h^{\vee}}
\newcommand{\so}{\mathfrak{so}}
\newcommand{\su}{\mathfrak{su}}
\newcommand{\iprod}[1]{\langle #1 \rangle}
\newcommand{\hroot}{\vartheta_{0}}
\newcommand{\p}{\rho}

\newcommand{\WW}{\mathcal{W}}
\newcommand{\sW}{s\mathcal{W}}
\newcommand{\EE}{\mathcal{E}}
\newcommand{\II}{\mathcal{I}}

\newcommand{\sgn}{\operatorname{sgn}}
\newcommand{\ceil}[1]{\left\lceil #1 \right\rceil}

\renewcommand{\AA}{\mathcal{A}}
\newcommand{\ZZ}{\mathcal{Z}}
\newcommand\jacobi[2]{{\genfrac(){}{0}{#1}{#2}}}

\newcommand{\bp}[1]{{^{\boxtimes #1}}}
\newcommand{\un}{{\operatorname{un}}}
\newcommand{\witt}[1]{\left[ {#1} \right]}

\def\Gab{\mbox{Gal}(\BQ^{\operatorname{ab}})}
\def\D{{\Delta}}
\def\F{{\Phi}}
\usepackage{diagbox}
\usepackage{graphics}

\newcommand{\FSexp}{\operatorname{FSexp}}
\newcommand{\sqdim}[1]{{\sqrt{\dim(#1)}}}
\renewcommand{\ss}[1]{{\varepsilon}_{#1}}
\newcommand{\HH}{\mathcal{H}}
\newcommand{\bt}{\boxtimes}
\newcommand{\bte}{{\boxtimes_{\EE}}}
\newcommand{\svec}{\operatorname{sVec}}
\newcommand{\smap}{\mathcal{S}}
\newcommand{\ot}{\otimes}

\newcommand{\MM}{\mathcal{M}}
\newcommand{\super}{\mathcal{S}}
\newcommand{\HS}{{\hat{S}}}
\newcommand{\Po}{{\Pi_{0}}}

\newcommand{\hN}{\hat{N}}

\newcommand{\FPdim}{\operatorname{FPdim}}
\newcommand{\sqFP}[1]{\sqrt{\FPdim(#1)}}
\newcommand{\NN}{\mathcal{N}}
\newcommand{\K}{\UU_2}

\newcommand{\esgn}{I_{\EE}}
\newcommand{\esgnp}{{I_{\EE}'}}
\newcommand{\ssp}[1]{{{\varepsilon}_{#1}'}}
\def\pt{{\operatorname{pt}}}

\def\GQ{\Gal(\overline{\BQ})}
\def\hs{{\hat{\s}}}
\def\perm{{\operatorname{Sym}}}

\calclayout

\title{Higher central charges and Witt groups}
\author{Siu-Hung Ng}
\address{Department of Mathematics\\
    Louisiana State University\\
    Baton Rouge, LA 70803\\
    U.S.A.}
\thanks{The first author was partially supported by the NSF grant DMS-1664418.}
\email{rng@math.lsu.edu}
\author{Eric C. Rowell}
\address{Department of Mathematics\\
    Texas A\&M University\\
    College Station, TX 77843-3368\\
    U.S.A.}
    \email{rowell@math.tamu.edu}
\thanks{The second author was partially supported by a Texas A\&M Presidential Impact Fellowship, a Simons Fellowship and  the NSF grant DMS-1664359.}
\author{Yilong Wang}
\address{Beijing Institute of Mathematical Sciences and Applications (BIMSA), Huairou, Beijing, China.}
\email{wyl@bimsa.cn}
\author{Qing Zhang}
\address{Department of Mathematics\\Purdue University\\
West Lafayette\\
IN 47907\\U.S.A.}
\email{zhan4169@purdue.edu}
\keywords{Higher central charge; Quantum group modular category; Signature; Witt group}

\begin{document}
\begin{abstract}
In this paper, we introduce the definitions of signatures of braided fusion categories, which are proved to be invariants of their Witt equivalence classes. These signature assignments define group homomorphisms on the Witt group. The higher central charges of pseudounitary modular categories can be expressed in terms of these signatures, which are applied to prove that the Ising modular categories have infinitely many square roots in the Witt group modulo the pointed part. This result is further applied to prove a conjecture of Davydov-Nikshych-Ostrik on the super-Witt group: the torsion subgroup generated by the completely anisotropic s-simple braided fusion categories has infinite rank.
\end{abstract}
\maketitle

\section{Introduction}

Fusion categories can be viewed as categorical generalizations of finite groups.  From this perspective
modular (tensor) categories are generalizations of metric groups $(G,q)$, i.e., finite abelian groups $G$ equipped with a non-degenerate quadratic form $q: G \to \BC^\times$.  Indeed, any pointed modular category (i.e., with each simple object tensor-invertible) can be constructed from a metric group, and conversely \cite{DGNO}.  Many structures and properties of metric groups can be generalized to the modular category setting.  Two important examples are Gauss sums and Witt equivalence.

The quadratic Gauss sum $\tau(G,q) = \sum_{x \in G} q(x)$ of a metric group $(G,q)$ is known to have the form $\zeta \sqrt{|G|} $ for some $8^{th}$-root of unity $\zeta$. In particular, the modulus of any Galois conjugate of the quadratic Gauss sum of a metric group is always equal to $\sqrt{|G|}$.  The categorical dimension $\dim(\CC)$ of a fusion category $\CC$ plays the role of the order of $G$, and
the counterpart of the quadratic Gauss sum for modular categories is the first Gauss sum $\tau_1(\CC)=\sum_{X\in\Irr(\CC)}(d_X)^2 \,\theta_X$ where $d_X$ are the categorical dimensions and $\theta_X$ are the twists.  While the categorical dimension $\dim(\CC)$ of a modular category $\CC$ may not be an integer, it is a totally positive cyclotomic integer (cf. \cite{ENO}). This means $\s(\dim(\CC))$ is a positive real cyclotomic integer for any automorphism $\s$ of $\overline{\BQ}$. It has been shown that the first Gauss sum $\tau_1(\CC)$ is equal to $\xi_1(\CC) \sqrt{\dim(\CC)}$ where $\xi_1(\CC)$ is a root of unity, and called the \emph{central charge of $\CC$} (cf. \cite{BakalovKirillov}). However, the Galois conjugates of the Gauss sum $\tau_1(\CC)$ can have moduli different from $\sqrt{\dim(\CC)}$. This apparent discrepancy between the Gauss sum in the categorical setting and the metric group setting inspired the notions of higher Gauss sums and higher central charges for modular categories introduced in \cite{NSW}.

The concept of Witt equivalence and the Witt group $\WW$ for non-degenerate braided fusion categories was introduced in \cite{DMNO}, generalizing the concept for metric groups. For metric groups, Witt equivalence is defined modulo groups with a hyperbolic quadratic form, where the operation is the usual direct product of metric groups. For non-degenerate braided fusion categories one uses the Deligne product $\boxtimes$ and considers equivalence classes modulo Drinfeld centers. It is worth noting that Witt classes do not depend on, or assume, any pivotal structure. Moreover, the classical Witt group $\WW_{\pt}$ corresponding to metric groups appears as a subgroup of $\WW$, as the Witt classes of pointed modular categories. Witt equivalence for slightly degenerate braided fusion categories was introduced in \cite{DNO}, and the corresponding Witt group  $\sW$ is called the \emph{super-Witt group} in this paper.  The study of the Witt group for non-degenerate braided fusion categories leads to many interesting questions about its structure (see \cite{DNO,DMNO}).  While it is known that the torsion subgroup $\Tor(\WW)$ of $\WW$ is a $2$-group with exponent $32$, it was not previously known whether $\Tor(\WW)/\WW_{\pt}$ has infinite cardinality or not.  Another interesting open problem is to find a set of generators for the Witt group $\WW$. One of the reasons these problems are difficult is that there are very few known invariants of the Witt group.  

Witt classes also have physical significance: symmetry gauging \cite{gauging,physicsgauging} and the reverse process boson condensation preserve Witt class.  Both of these are topological phase transitions in the theory of topological phases of matter \cite{physicsgauging,burnell}.  Each Witt class has a unique completely anisotropic representative \cite{DMNO}, and all members of that class can be reduced to this representative by anyon condensation.  Thus, distinguishing Witt classes can be regarded as analogous to determining allotropy classes in chemistry.

It is known \cite{ENO} that any pseudounitary braided fusion category has a unique canonical spherical structure so that the categorical dimensions $d_X>0$ for each $X$.  Moreover, the Witt classes with a pseudounitary representative form a subgroup $\WW_{\un}$ of $\WW$ \cite{DMNO}.  For a pseudounitary non-degenerate braided fusion category $\CC$ we may define its higher central charges to be those obtained from the modular category by endowing it with the canonical spherical structure yielding $d_X>0$ for all $X$. It has been proved in \cite{NSW} that the higher central charges of degrees coprime to the Frobenius-Schur exponents of any two Witt equivalent pseudounitary modular categories are equal, generalizing the case for the first central charge which was proved in \cite{DMNO}. In particular, the higher central charges are Witt invariants on $\WW_{\un}$. The higher central charges of $\CC$ with degrees coprime to the Frobenius-Schur exponent can be reformulated as a function $\Xi_\CC: \GQ \to \mu_\infty$, where $\mu_\infty \subset \BC$ is the group of roots of unity in $\BC$. This \emph{central charge function} $\Xi_\CC$ of $\CC$ can be expressed in terms of the first central charge $\xi_1(\CC)$, and the \emph{signature} $\ss{\CC} :\GQ \to \{\pm 1\}$ of $\CC$, which is a function given by $\ss{\CC}(\s) = \frac{\s(\sqrt{\dim(\CC)})}{|\s(\sqrt{\dim(\CC)})|}$. Since $\dim(\CC)$ is totally positive, $\sqrt{\dim(\CC)}$ is a totally real algebraic integer. Therefore, $\s(\sqrt{\dim(\CC)})$ is a nonzero real number, which can only be either positive or negative.  It is proved in Section 3 that the signature of a pseudounitary modular category is an invariant of its Witt class. Moreover, both the central charges and the signatures of pseudounitary modular categories can be extended to group homomorphisms $\Xi$ and $\varepsilon$ from the Witt subgroup $\WW_{\un}$  to the group $\U$ of functions from $\GQ$ to $\mu_\infty$ (cf. Section 4)

The signatures of nonzero totally real algebraic numbers, especially those of algebraic units, are studied in number theory \cite{DummitVoight, DDH}.  It is simple to determine the signatures of real quadratic numbers, but it is a difficult task for higher degree totally real algebraic number. However, for the quantum group modular categories $\CC_r:=\so(2r+1)_{2r+1}$ ($r \ge 1$), we are able to determine the signatures of a family of infinite subsequences of the $\CC_r$. Our starting point is a formula for $\sqrt{\dim(\CC_r)}$ expressed as a product of sine values of rational angles. The determination of the signatures of these modular categories is inspired by the computation  of the quadratic Gauss sum $\sum_{j = 0}^{k-1} \zeta_{k}^{j^{2}}$ in \cite{GaussDA} (see also \cite{Serre}).

Using these results, we study the Witt subgroups generated by the Witt classes of this family of subsequences. We prove in Theorem \ref{thm:signature} that the signatures of the categories in each of these subsequences are $\BF_2$-linearly independent functions, which enables us to determine the kernel and the image of the restriction of the signature homomorphism on these Witt subgroups. As a corollary, we show that they are all isomorphic to $\BZ/32\oplus (\BZ/2)^{\oplus \BN}$. Moreover, it is well-known that $[\CC_r]$ are square roots of the Ising modular categories in the Witt group, and we prove that the Ising modular category has infinitely many square roots in $\WW/\WW_{\pt}$.

In $\sW$, there is a subgroup $\sW_2$ generated by the Witt classes of completely anisotropic s-simple braided fusion categories of finite order, and $\sW_2$ is  of exponent 2 (cf. \cite{DNO}). It is conjectured  \cite[Conjecture 5.21]{DNO} that $\sW_2$ has infinite rank. By applying Corollary \ref{c:Gz} and Proposition \ref{prop:G-and-pt}, we prove this conjecture in Theorem \ref{thm:sW2}.

The paper is organized as follows: In Section 2, we briefly review the key concepts for fusion categories and the Witt group $\WW(\EE)$ over a symmetric fusion category $\EE$ for later use. In Section 3, we define two versions of signatures for fusion categories over $\EE$ and show that they are Witt invariants. In Section 4, we derive a formula of the higher central charges in terms of the categorical dimension signature, and we define the higher central charge homomorphism $\Xi$ on $\WW_{\un}$ in terms of the signature $\ss{}$. In Section 5, we study the quantum group modular categories $\so(2r+1)_{2r+1}$, with an emphasis on their signatures. In Section 6, we prove our first main results Theorem \ref{thm:signature}, Corollaries \ref{c:indep} and \ref{c:Gz}. We finally prove the conjecture of  Davydov-Nikshych-Ostrik in Theorem \ref{thm:sW2} of Section 7.

Throughout this paper, we use the notation $\zeta_a := \exp\left(\frac{2\pi i}{a}\right)$, and $\BQ_a := \BQ(\zeta_{a})$ for any $a \in \BN$. In particular,  $i := 
\exp\left(\frac{2\pi i}{4}\right) = \sqrt{-1}$.

\section{Preliminaries}
\subsection{Fusion categories and their global dimensions}\label{subsec:FC}
A \emph{fusion category} is a semisimple, $\BC$-linear abelian, rigid monoidal category with finite-dimensional Hom-spaces and finitely many simple objects which include the tensor unit $\1$ (cf. \cite{ENO}). For any fusion category $\CC$, we denote by $\Irr(\CC)$ the set of isomorphism classes of simple objects of $\CC$. If the context is clear, we will use the abuse notation to denote an object in an isomorphism class $X$ of $\CC$ by $X$. The tensor product endows $K_0(\CC)$, the Grothendieck group of $\CC$, a ring structure. More precisely, we have $X\o Y = \sum_{Z\in\Irr(\CC)} N_{X,Y}^{Z} Z$ for any $X, Y\in\Irr(\CC)$, where
$$N_{X,Y}^{Z} := \dim_{\BC}\CC(X\o Y, Z)\,.$$ 
For any $X\in\Irr(\CC)$, let $\NN_{X}$ be the square matrix of size $|\Irr(\CC)|$ such that $(\NN_{X})_{Y, Z} = N_{X,Y}^{Z}$ for any $Y, Z \in \Irr(\CC)$. The \emph{Frobenius-Perron dimension} (or FP-dimension) of $X\in\Irr(\CC)$, denoted by $\FPdim(X)$, is the largest positive eigenvalue for $\NN_{X}$ (cf. \cite{ENO}). The Frobenius-Perron dimension of $\CC$ is defined to be
$$
\FPdim(\CC) = \sum_{X \in \Irr(\CC)} \FPdim(X)^2\,.
$$
It is shown in \cite{ENO} that for any fusion category $\CC$, $\FPdim(\CC)$ is a totally positive cyclotomic integer.

Let $\CC$ be a fusion category. For any object $V$ of $\CC$, the left dual of $V$ is denoted by the triple $(V^*, \ev_V, \coev_V)$, where 
$\ev_{V}: V^* \o V \to \1$ and $\coev_{V}: \1 \to V \o V^{*}$ are respectively the evaluation and coevaluation morphisms for the left dual $V^*$ of $V$ (cf. \cite{tcat}). A simple object $X$ of $\CC$ is called \emph{invertible} if $X \ot X^* \cong \1$.  For any morphism $f: X \to X^{**}$, its \emph{left quantum trace} is defined to be 
$$
\Tr_X(f) := \ev_{X^{*}} \circ (f \o \id_{X^{*}}) \circ \coev_{X} \in \End_{\CC}(\1) \cong \BC\,.
$$

Since $\CC$ is a fusion category, $V \cong V^{**}$ for any $V \in \ob(\CC)$  (cf. \cite{mug2, ENO}). If $X \in \Irr(\CC)$ and  $h: X \to X^{**}$ is a nonzero morphism,  the \emph{squared norm} of $X$ is defined as
$$
|X|^2 := \Tr_X(h)\cdot \Tr_{X^{*}}((h^{-1})^*) \in \BC\,,
$$
which is independent of the choice of $h$. The \emph{global dimension} (or categorical dimension) of $\CC$ is defined as 
$$
\dim(\CC) = \sum_{X \in \Irr(\CC)} |X|^2\,.
$$
By \cite[Theorem 2.3]{ENO}, $|X|^2 > 0$ for any $X \in \Irr(\CC)$, and $\dim(\CC) \geq 1$ for any fusion category $\CC$. Moreover, by \cite[Remark 2.5]{ENO}, $\dim(\CC)$ is a totally positive algebraic integer. We will denote the positive square root of $\dim(\CC)$ by $\sqdim{\CC}$. 

One can extend the left duality of $\CC$ to a contravariant functor $(-)^*$. Then $(-)^{**}$ defines a monoidal functor on $\CC$.  A \emph{pivotal structure} on a fusion category $\CC$ is an isomorphism of monoidal functors $j: \id_\CC \xrightarrow{\cong} (-)^{**}$. A fusion category equipped with a pivotal structure is called a \emph{pivotal fusion category}, in which \emph{the quantum dimension} of any object $V \in \ob(\CC)$ is defined to be
$$
\dim_j(V) := \Tr_V(j_V) \in \End_{\CC}(\1) \cong \BC\,.
$$

A pivotal structure $j$ on $\CC$ is called \emph{spherical} if $\dim_j(V) = \dim_j(V^*)$ for all $V \in \ob(\CC)$. A pivotal fusion category $\CC$ is called a \emph{spherical fusion category} if its pivotal structure is spherical. We will simply denote the quantum dimension of an object $V \in \ob(\CC)$ by $d_V$ or $\dim(V)$ when the pivotal structure is clear from the context. 

If $\CC$ is a spherical fusion category,  then  $d_X^2 = |X|^2$ for any $X \in \Irr(\CC)$ (cf. \cite{mug2, ENO}). Therefore, $d_X \ne 0$ for $X \in \Irr(\CC)$ and 
$$
\dim(\CC) = \sum_{X \in \Irr(\CC)}d_X^2\,.
$$

By \cite[Proposition 5.7]{NS10}, if $\CC$ is a spherical fusion category, then $\dim(\CC) \in \BQ_{N}$, where $N = \FSexp(\CC)$ is the Frobenius-Schur exponent of $\CC$ (cf. \cite{NgSchaunburgSpherical}). If $\CC$ is not spherical, then we can consider the ``sphericalization'' $\tilde{\CC}$ of $\CC$ defined in \cite[Remark 3.1]{ENO}. In particular, $\tilde{\CC}$ is a spherical fusion category such that $\dim(\tilde{\CC}) = 2\dim(\CC)$ (cf. \cite[Proposition 5.14]{ENO}). The same argument as above implies that $\dim(\CC) \in \BQ_{\tilde{N}}$, where $\tilde{N} = \FSexp(\tilde{\CC})$. The upper bound $\tilde{N}$ can be determined by $\dim(\CC)$. In particular, $\dim(\CC)$ is a totally positive cyclotomic integer  for any fusion category $\CC$. 

\begin{remark}
The cyclotomicity of $\dim(\CC)$ can also be derived from \cite[Corollary 1.4]{codegree} and \cite[Theorem 8.51]{ENO}.
\end{remark}

A fusion category $\CC$ is called \emph{pseudounitary} if $\dim(\CC) = \FPdim(\CC)$. It is shown in \cite[Proposition 8.23]{ENO} that a pseudounitary fusion category $\CC$ admits a unique canonical spherical structure such that $d_V = \FPdim(V)$ for any $V \in \ob(\CC)$. In this paper, we assume that any pseudounitary fusion category is  equipped with its canonical spherical structure.

\subsection{Braided fusion categories and the square root of dimension}\label{subsec:BFC}

Let $\CC$ be a fusion category. A \emph{braiding} on $\CC$ is a natural isomorphism 
$$\b_{V, W}: V \ot W \xrightarrow{\cong} W \ot V$$ 
satisfying the Hexagon axioms (cf. \cite{tcat}). A fusion category equipped with a braiding is called a \emph{braided fusion category}. 

The degeneracy of a braiding $\beta$ on $\CC$ is characterized by double braidings. More precisely, let  $\CC'$ denote the \emph{M\"uger center} of $\CC$, which is the full subcategory $\CC$ determined by objects $V \in \ob(\CC)$ such that $\b_{W, V}\circ\b_{V, W} = \id_{V \ot W}$ for all $W \in \ob(\CC)$. The M\"uger center $\CC'$ is a fusion subcategory of  $\CC$.  A braided fusion category $\CC$ is called \emph{non-degenerate} if $\Irr(\CC') = \{\1\}$, i.e. $\CC'$ is equivalent to the category $\Vs$ of finite-dimensional vector spaces over $\BC$. From any fusion category $\AA$, one can construct a non-degenerate braided fusion category \cite{ JS, Kassel, Majid}, which is called the \emph{Drinfeld center} of $\AA$  and is denoted by $\ZZ(\AA)$.

In contrast to non-degenerate braided fusion categories, a braided fusion category $\CC$ is called a \emph{symmetric fusion category} if $\CC' = \CC$. For any finite group $G$, let $\Rep(G)$ be the fusion category of finite-dimensional complex representations of $G$, equipped with the usual braiding. If $z \in G$ is a central element of order 2, let $\Rep(G, z)$ be the fusion category $\Rep(G)$ equipped with the braiding given by the universal $R$-matrix $R=\frac{1}{2}(1\o 1 + 1 \o z +z \o 1 -z \o z)$. Both $\Rep(G)$ and $\Rep(G, z)$ are symmetric fusion categories, and the theorems of Deligne imply that any symmetric fusion category is braided equivalent to $\Rep(G)$ or $\Rep(G, z)$ for some finite group $G$ (cf. \cite{deligne1, deligne2}). A symmetric fusion category is called \emph{Tannakian} (resp.~\emph{super-Tannakian}) if it is braided equivalent to $\Rep(G)$ (resp.~$\Rep(G,z)$) for some finite group $G$. In particular, if $\CC$ is symmetric, then $\dim(\CC) \in \BZ$. 

In general, if $\CC$ is a braided fusion category, then $\CC'$ is either Tannakian or super-Tannakian. If $\CC'$ is braided equivalent to $\Rep(\BZ/2, 1)$, which is the category $\svec$ of finite-dimensional super-vector spaces over $\BC$, then $\CC$ is called \emph{slightly degenerate}.

A \emph{premodular} category is a spherical braided fusion category. A premodular category $\CC$ is called \emph{modular} if $\CC$ is non-degenerate. The  (unnormalized) \emph{ S-matrix} of a premodular category $\CC$ is defined to be
$$
S_{X,Y} := \Tr_{X^*\ot Y}(\b_{Y, X^*}\circ \b_{X^*, Y}),
\  X, Y \in \Irr(\CC)\,.
$$
An alternative criterion for  modularity of a premodular category is that the S-matrix is invertible (cf. \cite{MugerSubfactor2}).

Let $\CC$ be a premodular category. A natural isomorphism $\theta: \id_{\CC} \xrightarrow{\cong} \id_{\CC}$, called the \emph{ribbon structure} of $\CC$, can be defined using the spherical pivotal structure of $\CC$ and the Drinfeld isomorphism (cf. \cite{tcat}). The ribbon structure satisfies 
\begin{equation}\label{eq:twist-def}
\theta_{V\ot W} = (\theta_V \ot \theta_W)\circ \b_{W, V} \circ \b_{V, W}   
\end{equation}
and 
\begin{equation}
\theta_{V^{*}} = \theta_V^*
\end{equation}
for any $V, W \in \ob(\CC)$. In particular, for any $X \in \Irr(\CC)$, $\theta_X$ is equal to a non-zero scalar times $\id_X$. By an abuse of notation, we denote both the scalar and the isomorphism itself by $\theta_X$ for all simple $X$. The \emph{T-matrix} of a premodular category $\CC$ is defined to be the diagonal matrix
$$
T_{X,Y} := \theta_X\cdot \delta_{X, Y}\,,\,\quad X,Y\in\Irr(\CC).
$$
It is well-known that if $\CC$ is modular, then the S- and the T-matrices give rise to a projective representation of $\SL2$ (cf. \cite{BakalovKirillov, TuraevBook}). 

A premodular category $\CC$ is called \emph{super-modular} if $\CC$ is a slightly degenerate braided fusion category. The non-trivial simple object in $\CC'$, denoted by $f$, is an invertible object such that $\b_{f,f} = -\id_{f \o f}$.  This implies that $\theta_f d_f=-1$, so that $\theta_f=-d_f=\pm 1$. It is readily seen by \eqref{eq:twist-def} and the dimension equation $d_{f\ot X}= d_f\, d_X$ that tensoring with $f$ gives rise to a permutation on $\Irr(\CC)$ without any fixed point, and $X^* \not\cong f \o X$ for any $X \in \Irr(\CC)$.  Therefore, $\Irr(\CC)$ can be written as a disjoint union
$$
\Irr(\CC) = \Po \cup (f\ot\Po)
$$
for some  subset $\Po$ of $\Irr(\CC)$ containing $\1$ and closed under taking duals. With respect to this decomposition of $\Irr(\CC)$, the  S-matrix of $\CC$ takes the form
$$
S = 
\begin{pmatrix}
\hat{S} & d_f \hat{S}\\
d_f \hat{S} & \hat{S}
\end{pmatrix}\,.
$$

For any $X, Y \in \Po$, by \cite[Lemma 2.15]{Muger-Structure},  we have
\begin{equation}\label{eq:HS-sq}
\begin{aligned}
2\sum_{Z \in \Po} 
\HS_{X,Z}\HS_{Z,Y}
&=
\sum_{Z \in \Irr(\CC)}
S_{X,Z}S_{Z,Y}\\
&= 
\dim(\CC)\sum_{W \in \Irr(\CC')} N_{X,Y}^{W} d_W\\
&=
\dim(\CC)\delta_{ X, Y^{*} }\,,
\end{aligned}
\end{equation}
where the last equality is guaranteed by the assumption that $\Po$ is closed under taking duals. In particular, $\HS^2$ is a non-zero multiple of the charge conjugation matrix of $\Po$, and so $\HS$ is invertible.

Let $P$ denote the free abelian group over $\Po$.
For any $X, Y, Z \in \Po$, let
$$
\hN_{X,Y}^{Z} := N_{X,Y}^Z + d_f\cdot N_{X,Y}^{f\ot Z}\,.
$$
One can verify directly that the bilinear map $\bullet: P \times P \to P$ given by $X\bullet Y = \sum_{Z \in \Po} \hN_{X,Y}^{Z} Z$ defines a commutative ring structure on $P$ with the identity $\1$.  Moreover, by \cite[Lemma 2.4]{Muger-Structure}, for any $X, Y, Z \in \Po $, we have
\begin{equation}\label{eq:HS-char}
\begin{aligned}
\frac{\HS_{X,Y}}{d_Y}
\cdot
\frac{\HS_{Z,Y}}{d_Y}
&=
\sum_{W \in \Irr(\CC)} N_{X,Z}^{W} \frac{S_{W,Y}}{d_Y}\\
&=
\sum_{W \in \Po} \left(N_{X,Z}^{W} \frac{S_{W,Y}}{d_Y} + N_{X,Z}^{f\ot W} \frac{S_{f\ot W, Y}}{d_Y}\right)\\
&=
\sum_{W \in \Po} \left(N_{X,Z}^{W} \frac{\HS_{W,Y}}{d_Y} + d_f\cdot N_{X,Z}^{f\ot W} \frac{\HS_{W,Y}}{d_Y}\right)\\
&=
\sum_{W \in \Po} \hN_{X,Z}^{W} \frac{\HS_{W,Y}}{d_Y}\,.
\end{aligned}
\end{equation}
Therefore, the function $\chi_Y: P \to \BC$ defined by $\chi_Y(X) := \HS_{X,Y}/d_Y$ for $X \in \Po$ is a $\BC$-linear character of $P$. Now, \eqref{eq:HS-sq} implies that $\{\chi_Y: Y \in \Po\}$ is a set of $\BC$-linearly independent characters $P$, and hence it is the set of all the $\BC$-linear characters of $P$. Moreover, we have the following Verlinde-like formula
\begin{equation}\label{eq:super-Verlinde}
\hN_{X,Y}^{Z} 
= 
\frac{2}{\dim(\CC)}
\sum_{W \in \Po}
\frac{\HS_{X,W}\HS_{Y,W}\HS_{Z^{*},W}}{\HS_{\1, W}}\,.
\end{equation}

Recall that $S$ is a submatrix of the S-matrix of the Drinfeld center of $\CC$. It follows from \cite[Proposition 5.7]{NS10} that $S$ is a matrix defined over a certain cyclotomic field, and so is $\HS$.
The above discussion on the characters on $P$ implies that the absolute Galois group acts on $\Po$ by permutation. More precisely, let $\GQ$ be the absolute Galois group, then for any $\s \in \GQ$, and for any $Y \in \Po$, $\s(\chi_Y)$ is another character of $P$. Hence, there exists a unique $\hs(Y) \in \Po$ such that $\s(\chi_{Y}) = \chi_{\hs(Y)}$. Thus, for any $\s \in \GQ$ and for any $X, Y \in \Po$, we have
\begin{equation}\label{eq:HS-gal}
\s\left(
\frac{ \HS_{X, Y} }{ d_Y }
\right)
= \s(\chi_Y)(X)
=\chi_{\hs(Y)}(X)
=
\frac{ \HS_{X,\hs(Y)} }{ d_{\hs(Y)} }\,.
\end{equation}
Note that the above Galois property of a super-modular category is similar to that of a modular category (cf. \cite{dBG, CosteGannon94, ENO}).

The proof of the following lemma is identical to the proof for modular categories as in \cite{dBG, CosteGannon94, ENO}. However, we provide the proof for the sake of completeness.

\begin{lem}\label{lem:smod-gal}
Let $\CC$ be a super-modular category, and $D$  the positive square root of $\dim(\CC)$.
\begin{enumerate}
    \item For any $\s \in \GQ$, there exists a function $g_\s: \Po \to \{\pm 1\}$ such that for any $X, Y \in \Po$,
$$
\s\left(
\frac{ \HS_{X,Y} }{ D }
\right)
=
g_\s(X)
\frac{\HS_{\hs(X),Y} }{D }
=
g_\s(Y)
\frac{\HS_{X,\hs(Y)} }{D }\,.
$$
\item The positive real number $D$ is a cyclotomic integer.
\end{enumerate}

\end{lem}
\begin{proof}
Recall that for any $X \in \Po$, $d_X \neq 0$ (cf. Section \ref{subsec:FC}). By Eqs. (\ref{eq:HS-sq}) and (\ref{eq:HS-gal}), for any $\s \in \GQ$, we have
$$
\begin{aligned}
\s\left(
\frac{ \dim(\CC) }{d_X^2}
\right)
&=
\s\left(
2\sum_{Y \in \Po}
\frac{ \HS_{X,Y} \HS_{Y,X^{*}} }{ d_X^{2} }
\right)\\
&=
2\sum_{Y \in \Po}
\s\left(
\frac{ \HS_{X,Y} }{d_X}
\right)
\s\left(
\frac{ \HS_{ Y,X^{*} } }{d_{X^*}}
\right)\\
&=
2\sum_{Y \in \Po}
\frac{ \HS_{\hs(X),Y} }{d_{\hs(X)} }
\frac{ \HS_{Y,\hs(X^{*})} }{d_{\hs(X^{*})} }\\
&=
\delta_{\hs(X)^{*},\hs(X^{*}) }
\frac{\dim(\CC)}{d_{\hs(X)}^{2} }
\end{aligned}\,.
$$
Therefore, we have $\hs(X^{*}) = \hs(X)^{*}$ and
$$
\s\left(
\frac{D}{d_X}
\right)
=
g_{\s}(X) 
\frac{D}{d_{\hs(X)}}
$$
for some $g_\s(X) \in \{\pm 1\}$. Again by (\ref{eq:HS-gal}), for any $X, Y \in \Po$, we have 
$$
\s\left(\frac{ \HS_{X,Y} }{D}\right)
=
\s\left(
\frac{ \HS_{X,Y} }{d_Y}
\right)
\s\left(
\frac{ d_Y}{D}
\right)
=
g_\s(Y)
\frac{ \HS_{X, \hs(Y)} }{d_{\hs(Y)}}
\frac{d_{\hs(Y)}}{D}
=
g_\s(Y)\frac{\HS_{X,\hs(Y)}}{D}\,.
$$

Since $\dim(\CC)=D^2$ is an algebraic integer, so is $D$. Note that $\HS_{\1,\1}/D = 1/D$ and $\HS$ is symmetric. For any $\sigma, \tau \in \GQ$, 
$$
\sigma\tau\left(\frac{1}{D}\right) = g_\sigma(\1) g_\tau(\1)\left(\frac{\HS_{\hs(\1),\hat{\tau}(\1)}}{D}\right) = 
\tau \sigma\left(\frac{1}{D}\right)\,.
$$
Therefore, $\Gal(\BQ(D)/\BQ)$ is abelian.  By the Kronecker-Weber Theorem, $\BQ(D)$ is  contained in a cyclotomic field.  In particular, $D$ is a cyclotomic integer.
\end{proof}

Now we are ready to prove the following theorem.

\begin{thm}\label{thm:sqdim}
For any  pseudounitary braided fusion category $\CC$, $\sqdim{\CC}$ is a totally real cyclotomic integer.
\end{thm}
\begin{proof} Let $D= \sqdim{\CC}$. 
Since $\dim(\CC)$ is a totally positive algebraic integer, $D$ is a totally real algebraic integer. We are left to show the cyclotomicity of $D$. 



We have the following two cases.

(1) If $\CC'$ is Tannakian, then the pseudounitarity of $\CC$ implies $\theta_X = \id_X$ for $X \in \Irr(\CC')$. By \cite{Br, modclosure}, the de-equivariantization on $\CC$ with respect to $\CC'$ gives rise to a modular category $\MM(\CC)$ with
$$
\dim(\MM(\CC)) = \frac{\dim(\CC)}{\dim(\CC')}\,.
$$ 

By \cite[Theorem 7.1]{NS10}, $\sqdim{\MM(\CC)} \in \BQ_{12m}$ as $\FSexp(\MM(\CC)) \mid m$, where $m=\FSexp(\CC)$. Note that $\dim(\CC') \in \BZ$ as $\CC'$ is symmetric. Therefore,  $\sqdim{\CC'}$ is a cyclotomic integer, and so is 
$$
D = \sqrt{\dim(\MM(\CC))\dim(\CC')}\,.
$$

(2) If $\CC'$ is super-Tannakian, then $\CC'$ has a maximal fusion subcategory $\CC'_+$ which is Tannakian and
$$
\dim(\CC') = 2 \dim(\CC'_{+})\,.
$$
De-equivariantizing $\CC$ with respect to $\CC'_+$ gives rise to a super-modular category $\super(\CC)$ (cf. \cite[Section 5.3]{modclosure}) and
$$
\dim(\super(\CC)) = \frac{\dim(\CC)}{\dim(\CC'_+)}\,.
$$
By Lemma \ref{lem:smod-gal} (2), $\sqdim{\super(\CC)}$ is a cyclotomic integer. Using similar argument as in Case (1), $D$ is a cyclotomic integer.
\end{proof}
\begin{remark} (i)
The pseudounitary condition in the previous theorem could be removed but some technicality is required. However, this technicality can be circumvented if every super-modular category admits a minimal modular extension or every fusion category has a spherical structure.

\noindent (ii) In the proof of Theorem \ref{thm:sqdim}, the conductor of $\sqrt{\dim(\CC)}$ can be shown to be bounded by $12 \cdot \FSexp(\CC)$ if $\CC'$ is Tannakian by using the Cauchy Theorem \cite{BNRW}. It is unclear a similar bound can be obtained when $\CC'$ is super-Tannakian.
\end{remark}

\subsection{The Witt group $\WW(\EE)$}\label{subsec:WE} 
 In this section, we follow \cite{DNO} to study the Witt group of non-degenerate braided fusion categories over symmetric fusion categories. 

Let $\EE$ be a symmetric fusion category.  Throughout this paper, a \emph{fusion category over} $\EE$ is a fusion category $\AA$ equipped with a braided tensor functor $T_\AA: \EE \to \ZZ(\AA)$ such that the composition of $T_\AA$ and the forgetful functor $\ZZ(\AA) \to \AA$ is fully faithful.

 A tensor functor $F: \AA \to \BB$ between two fusion categories over $\EE$ is called a \emph{tensor functor over $\EE$} if $F$ is compatible with the embeddings $T_\AA$ and $T_\BB$. For details, see \cite[Section 2]{DNO}.

Let $\AA$, $\BB$ be two fusion categories over $\EE$, and $R: \EE\to\EE \bt \EE$ be the right adjoint functor to the tensor product functor $\ot: \EE\bt\EE\to\EE$. Then $A:=(T_{\AA}\bt T_\BB)R(\1)$ is a connected \'etale algebra in $\ZZ(\AA\bt\BB)$. The \emph{tensor product $\AA\bte\BB$ of $\AA$ and $\BB$ over $\EE$} is  defined to be $(\AA\bt\BB)_A$, the fusion category over $\EE$ of right $A$-modules. By \cite[Lemma 3.11]{DMNO}, we have 
\begin{equation}\label{eq:btedim}
\FPdim(\AA\bte\BB) 
=
\frac{\FPdim(\AA)\FPdim(\BB)}{\FPdim(A)}
=
\frac{\FPdim(\AA)\FPdim(\BB)}{\FPdim(\EE)}\,.
\end{equation}

Recall that the M\"uger center $\CC'$ of any braided fusion category $\CC$ is a symmetric fusion category. A braided fusion category $\CC$ equipped with a braided tensor equivalence $T: \EE \to \CC'$ is called a \emph{non-degenerate braided fusion category over $\EE$}. In particular, with this terminology, non-degenerate braided fusion categories are non-degenerate over $\Vs{}$, and slightly degenerate braided fusion categories are non-degenerate over $\svec$.

For any fusion category $\AA$ over $\EE$,  the M\"uger centralizer of $T_\AA(\EE)$ in $\ZZ(\AA)$ is denoted by $\ZZ(\AA, \EE)$, which is a typical example of non-degenerate braided fusion categories over $\EE$ (cf. \cite[Theorem 3.2]{Muger-Structure}, \cite[Theorem 3.10]{DGNO}). Since $\ZZ(\AA)$ is non-degenerate over $\Vs{}$, by \cite[Theorem 2.5]{ENO} and  \cite[Theorem 3.14]{DGNO}, we have
\begin{equation}\label{eq:dimZAE}
    \FPdim(\ZZ(\AA, \EE)) 
    = \frac{\FPdim(\ZZ(\AA))}{\FPdim(\EE)}
    = \frac{\FPdim(\AA)^2}{\FPdim(\EE)}\,.
\end{equation}

Two non-degenerate braided fusion categories $\CC$ and $\DD$ over $\EE$ are called \emph{Witt equivalent} if there exist fusion categories $\AA$ and $\BB$ over $\EE$ and a braided equivalence over $\EE$ such that
\begin{equation}\label{eq:witteq}
\CC\bte\ZZ(\AA, \EE) \cong \DD\bte\ZZ(\BB, \EE)\,.
\end{equation}

According to \cite{DNO}, the Witt equivalence is an equivalence relation among braided fusion categories over $\EE$, and the Witt equivalence classes form a group whose multiplication is given by $\bte$. We call this group the Witt group over $\EE$, and we denote it by $\WW(\EE)$. We denote the Witt class of a braided fusion category $\CC$ over $\EE$ by $[\CC]$. In case $\EE=\Vs{}$ or $\svec$, we simply denote by $\WW$ for $\WW(\Vs{})$ and $\sW$ for $\WW(\svec)$. The Witt group $\sW$ is also called the \emph{super-Witt group} in this paper.

By \cite[Proposition 5.13]{DNO}, the assignment
\begin{equation}\label{eq:s-map}
\smap: \WW \to \sW\,; \quad
[\CC] \mapsto [\CC\boxtimes\svec]
\end{equation}
is a group homomorphism, and it is shown in loc. cit. that  $\ker(\smap)$ is a cyclic group of order 16 generated by the class of any Ising braided category.  An Ising  category is a non-pointed fusion category $\CC$ of $\FPdim(\CC)=4$. There are 2 Ising categories up to tensor equivalence, and each of them admits 4 inequivalent braidings and they are all non-degenerate.  Since Ising categories are pseudounitary (cf. \cite{ENO}), these 8 inequivalent Ising braided categories are modular and they are classified by their central charges. 

It is shown in \cite{DNO} that the group $\WW$ has only 2-torsion, and the maximal finite order of an element of $\WW$ is 32. We have seen in the above paragraph that the classes of pseudounitary Ising modular categories are of order 16, but less is known about elements in $\WW$ of order 32. In Sections 6 and 7, we will show that the pseudounitary Ising modular categories have infinitely many square roots in $\WW$ modulo $\WW_\pt$. 

\section{The $\EE$-signatures of the Witt group $\WW(\EE)$}
Let $\a \ne 0$ be a totally real algebraic number. For each $\s \in \GQ$, $\s(\a)$ is either positive or negative. The sign of $\s(\a)$ is 1 if it is positive, and -1 otherwise. This assignment $\ss{}(\a)$ of signs
\begin{equation}\label{eq:sgn-a}
\ss{}(\a)(\s) := \sgn(\s(\a))
\end{equation}
for each $\s \in \GQ$ is called the \emph{signature} of $\a$. 

Let $\mu_n \subset \BC^\times$ denote the group of the $n^{th}$-roots of unity and $\mu_\infty = \bigcup_{n=1}^\infty \mu_n$. Then the set $\U$ of functions from $\GQ$ to $\mu_\infty$ is an abelian group under pointwise multiplication, and $\UU_n = \mu_n^{\GQ}$ is a subgroup of $\U$.  Thus, if $F$ is a totally real subfield of $\BC$,
$$\ss{}: F^\times \to \K$$ 
is a group homomorphism.

Recall that similar to $\FPdim(\CC)$, the well-definedness of the categorical dimension $\dim(\CC)$ of a fusion category $\CC$ does not depend on the existence of a pivotal structure on $\CC$ \cite{ENO, Muger-Structure}. Moreover, both dimensions are totally positive cyclotomic integers \cite{ENO}, so the positive square roots $\sqFP{\CC}$ and $\sqdim{\CC}$ are totally real.

\begin{defn} 
Let $\CC$ be a fusion category. We define the \emph{signature} $\ss{\CC}$ of $\CC$ as $\ss{}(\sqFP{\CC})$ and the \emph{categorical dimension signature} $\ssp{\CC}$ of $\CC$ as $\ss{}(\sqdim{\CC})$.
\end{defn}

By Theorem \ref{thm:sqdim}, for  pseudounitary braided fusion categories, we can change $\GQ$ to $\Gab$ in the definition of the categorical dimension signature. 

\begin{remark}\label{rmk:pseudo}
For any pseudounitary fusion category $\CC$, $\ss{\CC} = \ssp{\CC}$.  
\end{remark}

\begin{lem}
Let $\CC$, $\DD$ be fusion categories.
\begin{itemize}
\item[(a)] $\ss{\CC\bt\DD} = \ss{\CC}\cdot\ss{\DD}$, and $\ssp{\CC\bt\DD} = \ssp{\CC}\cdot\ssp{\DD}$, i.e., both signatures respect the Deligne tensor product of fusion categories.

\item[(b)] Both $\ss{\ZZ(\CC)}$ and $\ssp{\ZZ(\CC)}$ are the constant function 1.
\end{itemize}
\end{lem}
\begin{proof}
Statement (a) follows from the multiplicativity of the FP-dimension and the categorical dimension with respect to the Deligne tensor product.

Statement (b) follows from $\sqFP{\ZZ(\CC)} = \FPdim(\CC)$, $\sqdim{\ZZ(\CC)} = \dim(\CC)$ (cf. \cite{ENO}, \cite{MugerSubfactor2}) and the total positivity of both the FP-dimension and the categorical dimension. 
\end{proof}

\begin{thm} \label{t:signature}
For any symmetric fusion category $\EE$, the assignment 
$$
\esgn  : \WW(\EE) \to \K\,;\quad
[\CC] \mapsto \ss{\CC}\cdot\ss{\EE}
$$ 
is a well-defined group homomorphism. 
\end{thm}
\begin{proof}
We first show that the assignment $\esgn$ is well-defined. Indeed, for any braided fusion categories $\CC$ and $\DD$ over $\EE$ which are Witt equivalent over $\EE$, there exist fusion categories $\AA$, $\BB$ over $\EE$ such that $\CC\bte\ZZ(\AA, \EE) \cong \DD\bte\ZZ(\BB, \EE)$. Therefore, by Eqs. (\ref{eq:btedim}) and (\ref{eq:dimZAE}), we have
$$
\frac{\FPdim(\CC)\FPdim(\AA)^2}{\FPdim(\EE)^2}
=
\frac{\FPdim(\DD)\FPdim(\BB)^2}{\FPdim(\EE)^2}\,.
$$
This implies that $\sqFP{\CC}\FPdim(\AA) = \sqFP{\DD}\FPdim(\BB)$. As mentioned in the previous subsection, $\FPdim(\AA)$ and $\FPdim(\BB)$ are totally positive, so for any $\s \in \GQ$, we have
$$
\begin{aligned}
\ss{\CC}(\s) 
&= 
\sgn(\s\left(\sqFP{\CC}\right)) \\
&= 
\sgn(\s\left(\sqFP{\CC}\FPdim(\AA)\right))\\
&=
\sgn(\s\left(\sqFP{\DD}\FPdim(\BB)\right))\\
&= 
\sgn(\s\left(\sqFP{\DD}\right)) \\
&= 
\ss{\DD}(\s)
\end{aligned}
$$
which means $\ss{\CC} = \ss{\DD}$, and hence $\esgn$ is well-defined.

Again by (\ref{eq:btedim}), for any $\s\in \GQ$, we have
$$
\begin{aligned}
\esgn([\CC\bte\DD])(\s)
&=
\ss{\CC\bte\DD}(\s)\cdot \ss{\EE}(\s)\\
&= 
\sgn(\s\left(\sqrt{\frac{\FPdim(\CC)\FPdim(\DD)}{\FPdim(\EE)}}\right))\sgn(\s\left(\sqFP{\EE}\right))\\
&=
\sgn(\s\left(\sqrt{\FPdim(\CC)\FPdim(\DD)}\right))\\
&=
(\ss{\CC}\cdot\ss{\EE}\cdot\ss{\DD}\cdot\ss{\EE})(\s)\\
&=
(\esgn([\CC])\cdot \esgn([\DD]))(\s)
\end{aligned}
$$
as desired.
\end{proof}

\begin{thm}\label{thm:cat-dim-sgn}
For $\EE = \Vs$ or $\svec$, the assignment
$$
\esgnp: \WW(\EE) \to \K;\quad [\CC] \mapsto \ssp{\CC}\cdot\ssp{\EE}
$$
is a well-defined group homomorphism.
\end{thm}
\begin{proof}
By Theorem 2.5 and Theorem 3.10 (i) of  \cite{DGNO}, for any fusion category $\AA$ over $\EE$,
\begin{equation}\label{eq:cat-dim-ZAE}
\dim(\ZZ(\AA, \EE)) = \frac{\dim(\AA)^2}{\dim(\EE)}\,.
\end{equation}

For $\EE = \Vs$ or $\svec$, let $\AA$ and $\BB$ be fusion categories over $\EE$. In this case, there exists a finite group $G$ such that $\Rep(G)$ embeds into $\EE\bt\EE$ as a braided fusion subcategory. In fact, $G$ is trivial when $\EE = \Vs$, and $G = \BZ/2\BZ$ when $\EE = \svec$. Note that $|G| = \dim(\EE)$ in both cases. Moreover, the image of the regular algebra of $\Rep(G)$ under the composition $\Rep(G) \hookrightarrow \EE\bt\EE\xrightarrow{T_\AA\bt T_\BB} \ZZ(\AA\bt\BB)$ coincides with the algebra $A$ in the definition of the tensor product over $\EE$ (cf. Section \ref{subsec:WE}). Therefore, $\AA\bte\BB$ is braided equivalent to the de-equivariantization $(\AA\bt\BB)_G$, and by \cite[Proposition 4.26]{DGNO},  
\begin{equation}\label{eq:cat-dim-bt}
\dim(\AA\bte\BB) 
=
\frac{\dim(\AA)\dim(\BB)}{|G|}
=
\frac{\dim(\AA)\dim(\BB)}{\dim(\EE)}\,.
\end{equation}

Having established Eqs. (\ref{eq:cat-dim-ZAE}) and (\ref{eq:cat-dim-bt}), we are done by repeating the proof of Theorem \ref{t:signature} with all the FP-dimensions changed into categorical dimensions.
\end{proof}

\begin{remark}
(1) Our approach to prove \eqref{eq:cat-dim-bt} does not work for arbitrary symmetric fusion categories. If $\EE = \Rep(G)$, there may not be an embedding $\EE \to \EE\boxtimes\EE$ such that the regular algebra of $\EE$ coincides with the algebra $A \in \EE\boxtimes\EE$ used in the tensor product over $\EE$. For example, when $G$ is a non-abelian simple group, the only embeddings $\EE \to \EE\boxtimes\EE$ are $\EE\boxtimes \1$ and $\1\boxtimes\EE$. However, if \eqref{eq:cat-dim-bt} can be proved in general,  then the assignment $\esgnp$ in Theorem \ref{thm:cat-dim-sgn} is a well-defined group homomorphism. 

(2) One can also define $\esgnp$ on a ribbon version of the $\EE$-Witt group. Let $\EE_p$ be the ribbon category $\EE$ equipped with a spherical structure $p$. An \emph{$\EE_p$-modular category} $\CC$ is a ribbon category whose M\"uger center is equivalent to $\EE_p$ as ribbon categories. In particular, $\CC$ is an $\EE$-nondegenerate braided fusion category. Two $\EE_p$-modular categories $\CC$, $\DD$ are $\EE_p$-Witt equivalent if there exist spherical fusion categories $\AA$ and $\BB$ over $\EE_p$ such that $\CC \boxtimes_{\EE_{p}} \ZZ(\AA, \EE_{p}) $ and $ \DD \boxtimes_{\EE_{p}} \ZZ(\BB, \EE_{p})$ are equivalent as $\EE_p$-modular categories. The \emph{$\EE_p$-Witt group} $\WW_r(\EE_p)$ is the group of $\EE_p$-Witt equivalence classes. For any spherical fusion categories $\AA$ and $\BB$ over $\EE_p$, we have 
$$
\dim(\AA\boxtimes_{\EE_{p}}\BB) =\dim(\AA)\dim(\BB)/d_A =  \dim(\AA)\dim(\BB)/\dim(\EE)
$$
where $A$ is the algebra in the preceding remark (cf.~\cite{KiO}). In this case, $I'_\EE$ is a homomorphism of $\WW_r(\EE_p)$. However, we do not intend to pursue further discussion of this version of $I'_\EE$ in this paper.

\end{remark}

\begin{defn}
We call the group homomorphism $\esgn: \WW(\EE) \to \K$ defined in Theorem \ref{t:signature} the \emph{$\EE$-signature} on $\WW(\EE)$, and $\esgn([\CC])$ the $\EE$-signature of $[\CC]$. For $\EE = \Vs$ or $\svec$, we call the group homomorphism $\esgnp: \WW(\EE) \to \K$ defined in Theorem \ref{thm:cat-dim-sgn} the \emph{categorical dimension $\EE$-signature} on $\WW(\EE)$, and $\esgnp([\CC])$ the categorical dimension $\EE$-signature of $[\CC]$. 
\end{defn}

In practice, when there is a pseudounitary representative $\CC$ for a Witt class, the signatures $I_\EE([\CC])=I'_\EE([\CC])$ are essentially a function of the Galois group of $\BQ_n$, where $n$ is the conductor of $\sqrt{\dim(\EE) \dim(\CC)}$. For simplicity, $I_{\Vs{}}$, $I_{\svec}$, $I_{\Vs}'$ and $I_{\svec}'$ are denoted by $I$, $sI$, $I'$ and $sI'$ respectively. 

\begin{cor}\label{c:comm}
The following diagrams of group homomorphisms are commutative
\begin{equation*}
\begin{tikzcd}
\WW
\ar[rr, "\smap"]
\ar[rd, "I"']
&
&
\sW
\ar[ld, "sI"]
\\
&
\K
&
\end{tikzcd}\,,\quad
\begin{tikzcd}
\WW
\ar[rr, "\smap"]
\ar[rd, "I'"']
&
&
\sW
\ar[ld, "sI'"]
\\
&
\K
&
\end{tikzcd}\,.
\end{equation*}
\begin{flushright}
\end{flushright}
\end{cor}
\begin{proof}
The statement follows immediately from Theorems \ref{t:signature}, \ref{thm:sqdim} and the definition \eqref{eq:s-map} of $\smap$.
\end{proof}

\section{Higher central charges and signatures}

Let $\CC$ be a modular category. The $n^{th}$ Gauss sum $\tau_n(\CC)$ of $\CC$ introduced in \cite{NSW} is defined as
$$
\tau_n(\CC) = \sum_{X \in \Irr(\CC)} d_X^2 \theta_X^n\,.
$$
 If $\tau_n(\CC) \ne 0$, the $n^{th}$ central charge $\xi_n(\CC)$ is defined by
$$
\xi_n(\CC) = \frac{\tau_n(\CC)}{|\tau_n(\CC)|}\,.
$$
In particular, if $N$ is the Frobenius-Schur exponent of $\CC$ and $n$ is coprime to $N$,  by \cite[Theorem 4.1]{NSW}, $\tau_n(\CC) \ne 0$ and $\xi_n(\CC)$ is a root of unity. When there is no ambiguity, we simply write $\tau_n$ and $\xi_n$ for the $n^{th}$ Gauss sum and the $n^{th}$ central charge of $\CC$.

Recall that there is a group homomorphism $\hat{\cdot} :\GQ \to \perm(\Irr(\CC))$ from the absolute Galois group to the permutation group $\perm(\Irr(\CC))$ of $\Irr(\CC)$. By \cite[Proposition 4.7]{dong2015}, for any third root $\g$ of $\xi_1$, we have
\begin{equation}\label{eq:twist_galois}
\theta_{\hs(\1)} = \frac{\g}{\s^2(\g)}\,.
\end{equation}

The following theorem shows the relation between higher central charges of $\CC$ and the  signature of $\sqrt{\dim(\CC)}$ (cf. (\ref{eq:sgn-a})).
\begin{thm}\label{thm:higher-cc}
  Let $\CC$ be a modular category with Frobenius-Schur exponent $N$. Then for any integer $n$ coprime to $N$,
  $$
  \xi_n = \ssp{\CC}(\s) \cdot\s(\xi_1)\cdot \frac{\gamma^n}{\sigma^2(\gamma^n)}
  $$
  where $\gamma$ is any third root of $\xi_1$, $\sigma \in \GQ$ such that  $\sigma\inv(\zeta_N) = \zeta^n_N$.
\end{thm}
\begin{proof}
By  \cite[Theorem 4.1]{NSW},
$$
\t_n = \s(\t_1) \frac{\dim(\CC)}{\s(\dim(\CC))} \theta^n_{\hs(\1)}\,.
$$
Since $\dim(\CC)$ is totally positive and $\theta_{\hs(\1)}$ is a root of unity, we find
$$
|\t_n| = |\s(\t_1)| \cdot \frac{\dim(\CC)}{\s(\dim(\CC))}\,.
$$
Therefore,
$$
\xi_n =  \frac{\s(\t_1)}{|\s(\t_1)|} \theta^n_{\hs(\1)}\,.
$$
Since $\t_1 = |\t_1| \cdot \xi_1$ and $\dim(\CC) = |\t_1|^2$, we have
$$
\s(\t_1) = \s(D) \cdot \s(\xi_1),
$$
where $D=\sqrt{\dim(\CC)}$. Thus,
$$
|\s(\t_1)| = |\s(D)|\,.
$$
Therefore,
$$
\xi_n 
=  
\frac{\s(D)}{|\s(D)|} \s(\xi_1) \theta^n_{\hs(\1)}
=
\ssp{\CC}(\s)\cdot \s(\xi_1)\cdot \theta^n_{\hs(\1)}
\,.
$$ 
Now, the formula follows from \eqref{eq:twist_galois}.
\end{proof}
 
Note that since both $\s(\xi_1)$ and $\g/\s^2(\g)$ are completely determined by $\xi_1$ and $\s$, $\xi_n$ is completely determined by $\s$, $\xi_1$ and $\ssp{\CC}(\s)$. 
\begin{remark}
Consider $S$, the unnormalized S-matrix of a modular category $\CC$. Similar to Lemma \ref{lem:smod-gal}, there exists a sign function $\epsilon_\s: \Irr(\CC) \to \{\pm 1\}$ such that
$$
\s\left(\frac{S_{X,Y}}{D}\right) = \epsilon_\s(X) \frac{S_{\hs(X), Y}}{D} = \epsilon_\s(Y)  \frac{S_{X,\hs(Y)}}{D}
$$
for any $X, Y \in \Irr(\CC)$ (cf. \cite{tcat, dong2015}). In particular, $\s(S_{\1,\1}/D) = \epsilon_\s(\1) S_{\hs(\1), \1}/D$. This implies
$$
\s(D) = \epsilon_\s(\1) \frac{D}{d_{\hs(\1)}}\,.
$$
Therefore, 
$$
\ssp{\CC}(\s) = \sgn(\s(D)) =\epsilon_\s(\1) \cdot \sgn(d_{\hs(\1)})\,.
$$ 
If $d_X > 0$ for $X \in \Irr(\CC)$ or $\CC$ is pseudounitary, then $\ssp{\CC}(\s) = \epsilon_\s(\1)$.
\end{remark}

For the remaining discussion, it would be more convenient to define the higher multiplicative central charges of degrees coprime to the Frobenius-Schur exponent of $\CC$ as a function in $\U$. For any $N \in \BN$, and $k$ coprime to $N$, we use $\s_k$ to denote the element in $\Gal(\BQ_N/\BQ)$ such that $\s_k(\zeta_N) = \zeta_N^k$. 
\begin{defn}\label{def:charge-function}
Let $\CC$ be a modular category and $N=\ord(T_\CC)$. We define the higher central charge function $\Xi_\CC \in \U$ of $\CC$ as follows: for any $\s \in \GQ$, if $\s|_{\BQ_N} = \s_k$, then
$$
\Xi_\CC(\s) := \xi_k(\CC)\,.
$$
\end{defn}

In this convention, Theorem \ref{thm:higher-cc} can be restated as follows.

\begin{thm}\label{thm:higher-cc2}
Let $\CC$ be a modular category. Then for any $\s \in \GQ$,
  $$
  \Xi_\CC(\s) = \ssp{\CC}(\s\inv) \cdot\s\inv(\xi_1(\CC))\cdot \frac{\s(\gamma)}{\s\inv(\gamma)}
  $$
  where $\gamma$ is any third root of $\xi_1(\CC)$.
\end{thm}
\begin{proof}
The theorem is a direct consequence of Theorem \ref{thm:higher-cc} and Definition \ref{def:charge-function}.
\end{proof}

This formula of higher central charges allows us to define the function $\Xi$, in the following definition, on the subgroup $\WW_{\un}$  of $\WW$ generated by the pseudounitary modular categories. This function $\Xi$ will be shown to be a group homomorphism in the subsequent proposition.

\begin{defn}\label{def:charge_hom}
Let $\WW_{\un}$ be the subgroup of $\WW$ generated by the pseudounitary modular categories. The function $\Xi: \WW_{\un} \to \U$, called the \emph{higher central charge homomorphism}, is defined by
$$
\Xi([\CC])  =\Xi_\CC
$$
for any pseudounitary modular category $\CC$.
\end{defn}

\begin{prop}
 The higher central charge homomorphism $\Xi: \WW_{\un}\to \U$ is a well-defined group homomorphism.
\end{prop}
\begin{proof}
If $\CC$ and $\DD$ are Witt equivalent pseudounitary modular categories, then
$\xi_1(\CC)=\xi_1(\DD)$ and $\ssp{\CC}=\ssp{\DD}$ by Theorem \ref{thm:cat-dim-sgn}. Let $\g \in \BC$ be any $3^{rd}$ root of $\xi_1(\CC)$. Then, for any $\s \in \GQ$, we have
$$
\Xi_\CC(\s) = \ssp{\CC}(\s\inv) \cdot \s\inv(\xi_1(\CC)) \cdot\frac{\s(\gamma)}{\s\inv(\gamma)}=
\ssp{\DD}(\s\inv) \cdot \s\inv(\xi_1(\DD))\cdot \frac{\s(\gamma)}{\s\inv(\gamma)} = \Xi_\DD(\s)\,.
$$
By \cite[Lemma 3.1]{NSW} or Theorem \ref{thm:higher-cc2}, we also have
$$ \Xi_{\CC \boxtimes \DD}(\s) = \Xi_{\CC}(\s) \cdot \Xi_{\DD}(\s)$$ 
for any $\s \in \GQ$ and pseudounitary modular categories $\CC$, $\DD$.  Therefore, $\Xi$ is a group homomorphism and this completes the proof of the statement.
\end{proof}

\begin{remark}
Theorem \ref{thm:higher-cc2} also implies that the higher central charge homomorphism is equivalent to the product of the first central charge and the signature. More precisely, recall that the first central charge homomorphism is defined by $\xi_1: \WW_{\un}\to \mu_\infty$, $[\mathcal C]\mapsto \xi_1(\mathcal C)$ (see \cite{DMNO}).
Define a function $\mathfrak{F}: \mu_\infty \times \mathcal{U}_{2}\to \mathcal{U}_\infty$ by
 $$
 \mathfrak{F}(u, g)(\sigma):= 
 g(\sigma^{-1})\cdot
 \sigma^{-1}(u)  
 \cdot 
 \frac{\sigma(\gamma)}
 {\sigma^{-1}(\gamma)}\,,
 $$
 for any $\s\in \GQ$, where $\gamma$ is any third root of $u$ (note that the value of the function does not depend on the choice of $\gamma$).
 It is easy to see that $\mathfrak{F}$ is a group monomorphism, and Theorem \ref{thm:higher-cc2} implies the commutativity of the diagram:
 \begin{equation*}
      \begin{tikzcd}
\WW_{\un}
\ar[rr, "\xi_1\times I"]
\ar[rd, "\Psi"']
&
&
\mu_\infty\times \K
\ar[ld, "\mathfrak{F}"]
\\
&
\U
&
\end{tikzcd}\,.
\end{equation*}
In light of this commutative diagram, we will call $\Psi$ or the  equivalent map $\Psi':= \xi_1 \times I$ the higher central charge homomorphism.
\end{remark}

We close this section with the following proposition which will be useful for the last two sections, the proof of which follows immediately from the preceding remark.

\begin{prop} \label{p:ker}
The kernel of $\Xi$ consists of the Witt classes $[\CC] \in \WW_\un$ such that $\xi_1(\CC)=1$ and $\ss{\CC}$ is the constant function 1.\qed
\end{prop}


\section{The modular tensor categories $\so(2r+1)_{2r+1}$}
In this section, we provide some basic facts for the quantum group modular categories $$\CC_r := \so(2r+1)_{2r+1}$$ 
for $r \geq 1$. The readers are referred to \cite{BakalovKirillov, rowell2006} for more details on these categories. In particular, by \cite{Wenzl-Cstar}, $\CC_r$ is pseudounitary, so $\ss{\CC_{r}} = \ss{\CC_{r}}'$ for $r \in \BN$ by Remark \ref{rmk:pseudo}. We prove a formula for the higher central charges of $\CC_r$ in Lemma \ref{lem:ord-T} and a formula for $D_r = \sqrt{\dim(\CC_r)}$ in Proposition \ref{prop:Dr} which are essential to the proof of our major result.

\subsection{Notations and formulas}
Some basic facts of the categories $\CC_r$ can be extracted from the underlying Lie algebras $\so(2r+1)$ and their root/weight datum. The conventions and notations of roots, weights and information of $\CC_r$ are adopted from \cite{BakalovKirillov, BourbakiLie4, hump, rowell2006}. We list below some of the datum we will use in the next few sections. 

Let $n = 2r+1$ for some $r \geq 1$. The Lie algebra  $\so(n) = \so(2r+1)$ is of type $B_r$. We consider the quantum group modular category $\CC_r$ of $\so(n)$ at level $n$, and use the following notations for $r \ge 2$. 

\begin{itemize}
\item
    Orthonormal basis for the inner product space $(\BR^r, (\cdot\mid\cdot))$: $\{e_1, \ldots, e_r\}$.
\item Normalized inner product such that any short root $\a$ has squared length $2$:   $$\iprod{e_j,e_k} = 2\delta_{j,k} =2(e_j\mid e_k).$$

\item
    The set of positive roots: $\D_+$. Its contains the following elements
    $$
    \begin{aligned}
    e_j,& \ j = 1, \ldots, r;\\
    e_j-e_k,& \ 1 \leq j < k \leq r;\\
    e_j+e_k,& \ 1 \leq j < k \leq r\,.
    \end{aligned}
    $$
    In particular, $|\D_+| = r^2$.
    \item
    Fundamental weights:
    $$
    \begin{aligned}
    \w_j &= e_1 + \cdots + e_j,\ j = 1, \ldots, r - 1;\\
    \w_r &= \frac{1}{2}(e_1 + \cdots + e_r)\,.
    \end{aligned}
    $$
\item
    The set of dominant weights: $\F_+$.
\item
    Root lattice: $Q$. 
\item Coroot lattice: $$Q^\vee = \{\frac{2\a}{(\a\mid\a)}\mid \a\in Q \}\,.$$
    
\item
    Weight lattice: $P$. Note that the index of $Q^\vee$ in $P$ is given by
    \begin{equation}\label{eq:P-Q-v}
        |P/Q^{\vee}| = |P/Q|\cdot |Q/Q^{\vee}| = 4\,.
    \end{equation}
    
\item
    Half sum of positive roots:
    $$ 
    \p = \frac{1}{2}\Big((2r-1)e_1 + (2r-3)e_2 + \cdots + 3e_{r-1} + e_r\Big).
    $$
    
\item
    Highest root:
    $$
    \hroot = e_1 + e_2.
    $$
    
\item
    Dual Coxeter number:
\begin{equation}\label{eq:n-dcox}
    \dcox = 2r - 1 = n - 2\,.
\end{equation}
        
\item
    The fundamental alcove:
    \begin{equation}\label{eq:FA}
    \begin{aligned}
    C_r 
    &= \{\l\in \F_+\mid (\l+\p\mid \hroot) < n+\dcox \}\\
    &= \{\l\in \F_+\mid (\l+\p\mid \hroot) < 4r \}\\
    &= \{\l\in \F_+\mid (\l\mid \hroot) \le n \}.
    \end{aligned}
    \end{equation}
    Note that the isomorphism classes of simple objects of $\CC_r$ are indexed by $C_r$, and so we identify $C_r$ and $\Irr(\CC_r)$.
\item
    Quantum parameter:
    $$ 
    q = \exp{\left(\frac{\pi i}{2(n+\dcox)}\right)}
      = \exp{\left(\frac{\pi i}{4n-4}\right)}
      = \exp{\left(\frac{\pi i}{8r}\right)}\,.
    $$
    
\item
    Quantum integer:
    $$[m] = \frac{q^m-q^{-m}}{q - q^{-1}}\,.$$
    
\item
    Twist:
    $$
    \theta_\l = q^{2(\l\mid \l+2\p)} \quad\text{ for } \l \in C_{r}\,.
    $$
    
\item
    Quantum dimension:
    $$
    d_\l = \prod_{\a\in \D_{+}} \frac{[2(\l+\p\mid\a)]}{[2(\p\mid\a)]}\quad\text{ for } \l \in C_{r}\,.
    $$
    
\item
    First central charge:
    \begin{equation}\label{eq:xi-1}
    \begin{aligned}
    \xi_1(\CC_r) 
    &= \exp\left(\frac{2\pi i}{8}\cdot\frac{n\dim_{\BC}(\so(n))}{n+\dcox}\right)\\
    &= \exp\left(\frac{2\pi i}{8}\cdot\frac{(2r+1)\cdot(r(2r+1))}{4r}\right)\\
    &=
    \exp\left(2\pi i\frac{(2r+1)^2}{32}\right)\,.
    \end{aligned}
    \end{equation}
\end{itemize}

For $r=1$, $\CC_r=\so(3)_3$. All the above notations are the same except that $\vartheta_0=e_1$.


\subsection{Higher central charge of $\CC_{r}$}
Let $D_r = \sqdim{\CC_r}$ be the positive square root of $\dim(\CC_r)$. Let $N_r$ be the Frobenius-Schur exponent of $\CC_r$, and $T_r$ the T-matrix of $\CC_r$. By \cite[Theorem 7.7]{NgSchaunburgSpherical}, $$N_r = \ord(T_r) = \lcm\{\ord(\theta_{\l})\mid \l \in C_r\}\,.$$

\begin{lem}\label{lem:ord-T}
Let $r$ be a positive integer. Then
\begin{itemize}
\item[(a)] $\lcm\{32,4r\}\mid N_r \,\mid 32r$. 

\item[(b)] $D_r \in \BQ_{N_r}$, and for any $\s \in \GQ$, we have \begin{equation}\label{eq:higher-cc}
\Xi([\CC_r])(\s)
=
\ss{\CC_{r}}(\s\inv)\cdot\frac{\s(\xi_1(\CC_r))^{11}}{\s\inv(\xi_1(\CC_r))^{10}}\,.
\end{equation}
\end{itemize}
\end{lem}

\begin{proof}
Since $
2(\lambda \mid\lambda+2\p) \in \frac{1}{2}\BZ$, 
$\theta_\lambda= q^{2(\lambda \mid\lambda+2\p)}$ is a $32r^{th}$-root of unity. Therefore, $T_r^{32 r}=\id$ or $N_r \,\mid 32 r$. 

By (\ref{eq:FA}), we have $\w_r = \frac{1}{2}(e_1 + \cdots +e_r) \in C_r$. Therefore,
$$
2(\w_r\mid\w_r+2\p) = 
    \frac{r(2r+1)}{2}\,.
$$
Therefore,
$$
\theta_{\w_{r}} = q^{2(\w_{r}\mid \w_{r} + 2\p)} 
= 
\exp\left(\frac{\pi i}{8r} \cdot \frac{r(2r+1)}{2}\right) 
=
\exp\left(\frac{(2r+1)\pi i}{16}\right)\,,
$$
which implies that $\ord(\theta_{\w_{r}}) = 32$. Thus, it suffices to consider $r \ge 3$ for statement (a). 

Note that by (\ref{eq:FA}), we have $2 e_1 \in C_r$ for $r \ge 3$. We have
$$
2(2 e_1\mid 2 e_1+2\p) = 8 + 4(2r-1) = 8r+4\,.
$$
Therefore,
$$
\theta_{2e_{1}} 
= 
\exp\left(\frac{\pi i}{8r} \cdot (8r+4)\right)
=
-\exp\left(\frac{\pi i}{2r} \right)
$$
which implies that $\ord(\theta_{2e_{1}}) = 4r$. The above computations imply the first divisibility of statement (a). 

By (\ref{eq:xi-1}), $\g_r = \xi_1(\CC_r)^{11}$ is a $3^{rd}$ root of $\xi_1(\CC_r)$, and $\g_r^{32}=1$. By (a), $\g_r \in \BQ_{N_r}$.  This implies that $D_r \in \BQ_{N_{r}}$ by \cite[Theorem II (ii)]{dong2015}. The remaining statement of (b) follows immediately from Theorem \ref{thm:higher-cc2}.
\end{proof}

\begin{remark}
The preceding lemma implies that  $\ord(T_r)$ is non-decreasing with respect to $r$.
\end{remark}


\subsection{A formula for $D_r$}
\begin{prop}\label{prop:Dr}
The square root of the global dimension of $\CC_r$ is given by 
\begin{equation}\label{eq:Dr}
D_{r} = \frac{\sqrt{r^r}}{2^{r^{2}-r-1}} 
\left(
\prod_{\ell = 1}^{r} \sin\left(\frac{(2\ell-1)\pi}{8r}\right)
\prod_{j = 1}^{2r-2}
\sin\left(\frac{j\pi}{4r}\right)^{m_{r}(j)}
\right)^{-1}\,
\end{equation}
where $m_r(j)$,    $1 \leq j \leq 2r-2$, is given by
$$
m_r(j) = 
\begin{cases}
0,\quad \mathrm{if}\  r = 1\,;\\
r - \ceil{\frac{j}{2}},\quad \mathrm{if}\ r \geq 2\,.
\end{cases}
$$
\end{prop}

\begin{proof}
According to \cite[Theorem 3.3.20]{BakalovKirillov}, 
\begin{equation}\label{eq:Dr-BK}
D_{r} 
= 
\sqrt{|P/(n+\dcox)Q^{\vee}|}
\prod_{\a \in \D_{+}} 
\left(
2\sin\left(\frac{(\a\mid\p)}{n+\dcox}\cdot\pi\right)
\right)^{-1}\,.
\end{equation}
By (\ref{eq:P-Q-v}), (\ref{eq:n-dcox}) and the fact that $|\D_+| = r^2$, we have
\begin{equation}\label{eq:Dr-simple}
\begin{aligned}
D_{r} 
&=
\frac{\sqrt{4(4r)^r}}{2^{r^{2}}}
\prod_{\a \in \D_{+}} 
\left(
\sin\left(\frac{(\a\mid\p)}{4r}\pi\right)
\right)^{-1}\\
&=
\frac{\sqrt{r^r}}{2^{r^{2}-r-1}}
\prod_{\a \in \D_{+}} 
\left(
\sin\left(\frac{(\a\mid \p)}{4r}\pi\right)
\right)^{-1}\\
\end{aligned}
\end{equation}

Recall that $\D_+ = \{e_\ell, e_a\pm e_{b}\mid 1 \leq \ell \leq r, 1 \leq a<b \leq r\}$, then $(\a\mid \p)$ can be easily given as follows: 

When $\a = e_\ell$ for some $1 \leq \ell\leq r$, $(\a\mid \p)= (r - \ell) + 1/2$ which is a half integer;
when $\a = e_a + e_{b}$ for some $1 \leq a< b \leq r$, $(\a\mid \p)= 2r - (a + b) + 1$ which is an  integer satisfying
$$
2 \leq (\a\mid \p)  \leq 2r-2\,;
$$
when $\a = e_a - e_{b}$ for some $1 \leq a < b \leq r$, $(\a\mid \p)= b - a$ which is also an integer satisfying
$$
1 \leq (\a\mid \p)  \leq r - 1\,.
$$

Let $m_r(j)=\#\{\a \in \D_+ \mid (\a \mid \rho) =j\}$ for $1 \le j \le 2r-2$.  We can now rewrite (\ref{eq:Dr-simple}) as
\begin{equation*}
D_{r} 
=
\frac{\sqrt{r^r}}{2^{r^{2}-r-1}} 
\left(
\prod_{\ell = 1}^{r} \sin\left(\frac{(2\ell-1)\pi}{8r}\right)
\prod_{j = 1}^{2r-2}
\sin\left(\frac{j\pi}{4r}\right)^{m_r(j)}
\right)^{-1}
\end{equation*}

When $r = 1$, $(\a\mid \rho)$ is not an integer for $\a \in \D_+$ and hence the lemma follows directly from \eqref{eq:Dr-simple}. We proceed to show that $m_r(j) = r - \ceil{\frac{j}{2}}$ for $r \geq 2$ and $1 \le j \le 2r-2$ by induction on $r$. 

Note that 
the equations $j = b-a$ and $j = 2r - (b+a)+1$ have no common integer solution $(a,b)$ for any integer $j$. Thus
$$
    m_r(j)  = |M_r(j)|
$$
for $1\le j \le 2r-2$, where 
$$
M_r(j):= \{(a,b)\mid 1 \le a < b \le r \text{ such that }  (b-a -j)(2r - (b+a)+1 - j) =0 \}\,.
$$
One can check directly that $m_2(j)=2-\ceil{\frac{j}{2}}$ for $1\le j \le 2$. Assume that $m_r(j) = r - \ceil{\frac{j}{2}}$ for all $1\le j \le 2r-2$ for some integer $r$.  Note that if $1\le j \le 2r-2$, then 
$$
M_r(j)+(1,1)= \{(a,b) \in M_{r+1}(j) \mid a \ge 2 \}\,.
$$
Thus,
$$
M_{r+1}(j) = \left\{
\begin{array}{ll}
    (M_r(j)+(1,1)) \cup \{(1, j+1)\} & \text{ if } 1 \le j \le r,   \\ \\
    (M_r(j)+(1,1)) \cup \{(1, 2r+2-j)\} & \text{ if } r+1 \le j \le 2r-2\,. 
\end{array}\right.
$$
Therefore, $m_{r+1}(j) = m_r(j)+1$ for $1\le j \le 2r-2$.
It is easy to see that
$$
M_{r+1}(2r-1) =\{(1,3)\}\,,   M_{r+1}(2r) = \{(1,2)\} \text{ and } \ceil{\frac{2r-1}{2}} =\ceil{\frac{2r}{2}} = r\,.
$$
Thus,
$$
m_{r+1}(j)=|M_{r+1}(j)| = 1 = r +1 - \ceil{\frac{j}{2}}
$$
for $2r-1 \le j \le 2r$. Therefore, we have $m_{r+1}(j)= r+1 -\ceil{\frac{j}{2}}$ for $1 \le j \le 2r$. 
\end{proof}


\section{Witt subgroups generated by $\CC_r$}\label{sec:the-group}
Let $\II := \su(2)_2$ be a fixed Ising modular category (cf. \cite{DGNO}). It is well-known (cf. \cite{DMNO}) that for any $n = 2r+1$, the conformal embedding $\so(n)_n \times \so(n)_n \subset \so(n^2)_1$ implies $\CC_r \boxtimes \CC_r$ is Witt equivalent to an Ising modular category. By comparing the first central charges, the Witt class $[\CC_r]$ of $\CC_r$ satisfies
\begin{equation}\label{eq:xi-of-Ising}
[\CC_r]^2 = [\CC_r\bp{2}] 
=
\begin{cases}
[\II]^{11} \quad \mathrm{if}\  r \equiv 0 \ \mathrm{or} \ 3 \pmod 4, \\
[\II]^3, \quad \mathrm{if}\  r \equiv 1 \ \mathrm{or} \ 2 \pmod 4 .
\end{cases}
\end{equation}
 Since the first central charge of $\CC_r$ is a primitive $32^{nd}$ root of unity (cf.  \eqref{eq:xi-1}), the cyclic subgroup  $\langle [\CC_r] \rangle$ of $\WW$ generated by $[\CC_r]$ is of order 32 for any positive integer $r$.

In this section, we study the subgroup of the Witt group $\WW$ generated by the Witt classes $[\CC_r]$ for $r \geq 1$ using their higher central charges. We proceed by investigating some  number-theoretical properties of them. 


\subsection{The signature of $\CC_r$} In this subsection, we compute some values of the signatures of an infinite subset of $\{\CC_r \mid r \in \BN\}$.
\begin{lem}\label{lem:gal-sin} 
For any integer $r, j, k$ with $k \equiv 1 \pmod 4$ and $\gcd(k, r) = 1$, 
\begin{equation}
   \s_k\left(\sin\left( \frac{j\pi}{8r} \right)\right)= \sin\left( \frac{k j\pi}{8r}  \right) \quad\text{ in } \BQ_{16r}\,.
\end{equation}
\end{lem}
\begin{proof}
$$
\s_k
\left(
\sin\left(\frac{j\pi}{8r}\right)
\right)
=
\s_k\left(
\frac{e^{\frac{j\pi i}{8r}} - e^{\frac{-j\pi i}{8r}}}
{2i}
\right)
=
\left(
\frac{e^{\frac{k j\pi i}{8r}} - e^{\frac{-k j\pi i}{8r}}}
{2i}
\right)
=
\sin\left(
\frac{k j\pi}{8r}
\right)\,.\qedhere
$$
\end{proof}

\begin{lem} \label{l:jac}
  Let $r$ be an odd positive integer. For any integer $k$ coprime to $r$, 
  $$
  \s_k(\sqrt{r^*}) =\jacobi{k}{r}\sqrt{r^*},
  $$
  where $r^* = \jacobi{-1}{r} r$, and $\jacobi{\bullet}{r}$ is the Jacobi symbol. If, in addition, $k \equiv 1 \pmod{4}$, then,  in $\BQ_{4r}$,  we have
  $$
  \s_k(\sqrt{r}) =\jacobi{k}{r}\sqrt{r}\,.
  $$
\end{lem}
\begin{proof}
For any prime factor $p$ of $r$, $\sqrt{p^*} \in \BQ_p \subseteq \BQ_r$ (cf. \cite{washington}). Note that $\Gal(\BQ_p/\BQ)$ is cyclic of order $p-1$. Thus, if  $a$ is a primitive root of $(\BZ/p)^\times$, then $\s_a$ is a generator of $\Gal(\BQ_p/\BQ)$ and $\s_a(\sqrt{p^*}) = -\sqrt{p^*}$. Therefore, $\s_a^j(\sqrt{p^*}) = (-1)^j \sqrt{p^*}$ for any integer $j$. Hence, we have $\s_k(\sqrt{p^*})=\jacobi{k}{p}\sqrt{p^*}$. 

If $r=p_1\cdots p_\ell$ is the prime factorization of $r$, then 
$r^* = p_1^* \cdots p_\ell^*$ and so
$$
\s_k(\sqrt{r^*}) = \s_k(\sqrt{p_1^*}) \cdots \s_k(\sqrt{p_\ell^*}) = \jacobi{k}{p_1} \cdots \jacobi{k}{p_\ell}\sqrt{r^*} = \jacobi{k}{r}\sqrt{r^*}\,.
$$
If, in addition, $k \equiv 1 \pmod 4$, then, in $\BQ_{4r}$, we have 
$$
\s_k(\sqrt{r}) = \s_k\left(\sqrt{r^*} \sqrt{\jacobi{-1}{r}}\right) = \jacobi{k}{r} \sqrt{r^*}\sqrt{\jacobi{-1}{r}} = \jacobi{k}{r} \sqrt{r}. \qedhere
$$
\end{proof}

\begin{prop}\label{prop:sign-D-2-8}
For any integers $l > 0$, $w > 0$ and $y$, let $a = 2 l+1 +w (8 l+2)$ and $k = 8ya + 4l+1$. Then, $\gcd(k, 4a)=1$ and $$\ss{\CC_{a}}(\s_{k})=(-1)^{y}.$$ 
\end{prop}

\begin{proof}
The first assertion follows directly by the Euclidean algorithm. We proceed to compute the sign of each component of $\s_{k}(D_a)$ in the right hand side of (\ref{eq:Dr}). The sign of the second sine component is 1 by the following lemma.
\begin{lem}\label{lem:even-sine-2-8} We have the following equality:
\begin{equation*}
\sgn\left(\s_{k}
\left(
\prod_{j = 1}^{2a-2}
\sin\left(\frac{j\pi}{4a}\right)^{m_a(j)}
\right)\right) = 1\,.
\end{equation*}
\end{lem}

\begin{proof}[Proof of Lemma \ref{lem:even-sine-2-8}]
Since $k \equiv 1 \pmod 4$, by Lemma \ref{lem:gal-sin}, we have
\begin{equation*}
\begin{aligned}
\sgn\left(\s_{k}
\left(
\prod_{j = 1}^{2a-2}
\sin\left(\frac{j\pi}{4a}\right)^{m_a(j)}
\right)\right)
=
\prod_{j = 1}^{2a-2}
\sgn\left(
\sin\left(
\frac{kj\pi}{4a}
\right)\right)^{m_a(j)}\,.
\end{aligned}
\end{equation*}
For each $j = 1, \ldots, 2a - 2$, we have
\begin{equation*}
\begin{aligned}
\sin\left(\frac{kj\pi}{4a}\right)
=
\sin\left(\frac{(8ya+4l+1)j\pi}{4a}\right)
=
\sin\left(\frac{(4l+1)j\pi}{4a}\right)\,.
\end{aligned}
\end{equation*}
Moreover, by the definition of $a$ and the assumption that $w > 0$, we have
\begin{equation*}
\frac{(4l+1)(2a-2)}{4a} - 2l
> 0\,.
\end{equation*}
Therefore,
\begin{equation*}
2l 
< 
\frac{(4l+1)(2a-2)}{4a} 
< 
\frac{(4l+1)(2a)}{4a} 
= 
2l+\frac{1}{2}\,.
\end{equation*}
Consequently, for $j = 1, \ldots, 2a - 2$, $\sgn\left(\sin\left(\frac{(4l+1)j\pi}{4a}\right)\right) = -1$ if and only if
\begin{equation}\label{eq:even-sine-range-2-8}
2q-1 < \frac{(4l+1)j}{4a} < 2q
\end{equation}
for some $1 \leq q \leq l$.

For any $q = 1, \ldots, l$, (\ref{eq:even-sine-range-2-8}) is equivalent to
\begin{equation*}
\begin{aligned}
(8w+2)(2q-1) + \frac{4q-2}{4l+1}
 < 
j
<
(8w+2)(2q) + \frac{4q}{4l+1}\,.
\end{aligned}    
\end{equation*}
Since $1 \leq q \leq l$, we have
\begin{equation*}
0 < \frac{4q-2}{4l+1} < \frac{4q}{4l+1} < 1\,.
\end{equation*}
So (\ref{eq:even-sine-range-2-8}) is equivalent to \begin{equation*}
(8w+2)(2q-1)+1 \le j \leq (8w+2)(2q)\,.    
\end{equation*}
There are exactly $8w+2$ integers between $(8w+2)(2q-1)+1$ and $(8w+2)(2q)$ inclusively. They can be written in pairs $(8w+2)(2q-1) + (2t-1), (8w+2)(2q-1) + 2t$ for $1 \leq t \leq 4w+1$. However, for each such $t$, we have
\begin{equation*}
\begin{aligned}
&
m_a((8w+2)(2q-1) + (2t-1))\\ 
=& 
a - \ceil{\frac{(8w+2)(2q-1)+(2t-1)}{2}}\\
=&
a - ((4w+1)(2q-1) + t)\\
=&
a - \ceil{\frac{(4w+1)(2q-1)+(2t)}{2}}\\
=&
m_a((8w+2)(2q-1)+ (2t))\,.
\end{aligned}
\end{equation*}
Thus we have
\begin{equation}\label{eq:even-sine-2-8}
\begin{aligned}
&
\prod_{j = 1}^{2a-2}
\sgn\left(
\sin\left(
\frac{kj\pi}{4a}
\right)\right)^{m_a(j)}\\
=&
\prod_{q = 1}^{l}\prod_{t=1}^{4w+1} 
(-1)^{m_a((8w+2)(2q-1) + (2t-1))} 
\cdot
(-1)^{m_a((8w+2)(2q-1) + 2t)}\\
=& 
1\,.
\end{aligned}
\end{equation}
as desired.\qedhere
\end{proof}
 The sign of the first sine component of the right hand side of \eqref{eq:Dr} is computed in the following lemma.
\begin{lem}\label{lem:odd-sine-2-8} We have the following equality:
\begin{equation*}
\sgn
\left
(\s_{k}
\left(
\prod_{j= 1}^{a} \sin
\left
(\frac{(2j-1)\pi}{8a}
\right)
\right)
\right) 
= (-1)^{y+l}\,.
\end{equation*}
\end{lem}

\begin{proof} [Proof of Lemma \ref{lem:odd-sine-2-8}]
Since $k \equiv 1 \pmod 4$, by Lemma \ref{lem:gal-sin}, we have
\begin{equation*}
\s_{k}
\left(
\prod_{j = 1}^{a}
\sin\left(\frac{(2j-1)\pi}{8a}\right)
\right)
=
\prod_{j = 1}^{a}
\sin\left(
\frac{k(2j-1)\pi}{8a}
\right)\,.
\end{equation*}
Thus, 
\begin{equation*}
\sgn\left(\s_{k}
\left(
\prod_{j = 1}^{a}
\sin\left(\frac{(2 j-1)\pi}{8a}\right)
\right)\right)
=
\prod_{j = 1}^{a}
\sgn\left(
\sin\left(
\frac{k(2j-1)\pi}{8a}
\right)\right)\,.
\end{equation*}
Moreover, by definition, we have
\begin{equation*}
\begin{aligned}
\sin\left(
\frac{k(2j-1)\pi}{8a}
\right)
&=
\sin\left(
\frac{(8ya+4l+1)(2j-1)\pi}{8a}
\right)\\
&=
(-1)^y\sin
\left(
\frac{(4l+1)(2j-1)\pi}{8a}
\right)\,.
\end{aligned}
\end{equation*}
Therefore, since $a$ is odd, we have
\begin{equation}\label{eq:odd-sine-decomp}
    \sgn\left(\s_{k}
\left(
\prod_{j = 1}^{a}
\sin\left(\frac{(2j-1)\pi}{8a}\right)
\right)\right)
=
(-1)^y
\prod_{j = 1}^{a}
\sgn\left(
\sin\left(
\frac{(4l+1)(2j-1)\pi}{8a}
\right)\right)\,.
\end{equation}

Since
\begin{equation*}
\frac{(4l+1)(2a-1)}{8a} - l  > 0\,,
\end{equation*}
we have
\begin{equation*}
l < \frac{(4l+1)(2a-1)}{8a} < \frac{(4l+1)(2a)}{8a} = l + \frac{1}{4}\,.
\end{equation*}

Now we consider two cases.

\textbf{Case 1.} If $l$ is even, then for any $j = 1, \ldots, a$, $\sgn(\sin(\frac{(4l+1)(2j-1)}{8a} \pi )) = -1$ if and only if
\begin{equation}\label{eq:odd-sine-even-l}
2q-1 < \frac{(4l+1)(2j-1)}{8a} < 2q    
\end{equation}
for some $1 \leq q \leq l/2$. Note that for any $q = 1, \ldots, l/2$, (\ref{eq:odd-sine-even-l}) is equivalent to 
\begin{equation*}
(8w+2)(2q-1) + \frac{4q-2}{4l+1} + \frac{1}{2}
<
j 
< 
(8w+2)(2q) + \frac{4q}{4l+1} + \frac{1}{2}\,.
\end{equation*}
Since $1 \leq q \leq l/2$, we have
\begin{equation*}
0 < \frac{4q-2}{4l+1} < \frac{4q}{4l+1} < \frac{1}{2}\,,    
\end{equation*}
(\ref{eq:odd-sine-even-l}) is equivalent to 
\begin{equation*}
(8w+2)(2q-1) + 1 \leq j \leq (8w+2)(2q) = 16wq+4q\,.
\end{equation*}
Hence, we have
\begin{equation*}
\prod_{j = 1}^{a}
\sgn\left(
\sin\left(
\frac{(4l+1)(2j-1)\pi}{8a}
\right)\right)\\
=
\prod_{q = 1}^{l/2} \left(\prod_{j = (8w+2)(2q-1)+1}^{16wq+4q}
(-1) \right)=1 
\end{equation*}
as there are $8w+2$ terms in each of the product corresponding to $q$. 

\textbf{Case 2.} If $l$ is odd, then for any $j = 1, \ldots, a$, $\sgn(\sin(\frac{(4l+1)(2j-1)}{8a} \pi )) = -1$ if and only if $j$ satisfies
\begin{equation}\label{eq:odd-sine-odd-l-range-1}
2q-1  < \frac{(4l+1)(2j-1)}{8a} < 2q    
\end{equation}
for some $1 \leq q \leq (l-1)/2$, or $j$ satisfies
\begin{equation}\label{eq:odd-sine-odd-l}
l < \frac{(4l+1)(2j-1)}{8a} < l+\frac{1}{4}\,.
\end{equation}
By the same argument as in Case 1, the sign for the $j$'s satisfying (\ref{eq:odd-sine-odd-l-range-1}) is equal to
\begin{equation*}
\prod_{q = 1}^{(l-1)/2} \left(\prod_{j = (8w+2)(2q-1)+1}^{16wq+4q}
(-1)\right) = 1\,.
\end{equation*}

Note that (\ref{eq:odd-sine-odd-l}) is equivalent to 
\begin{equation*}
8wl + 2l + \frac{2l}{4l+1} + \frac{1}{2}
<
j
<
a + \frac{1}{2}\,.
\end{equation*}
Since $\frac{2l}{4l+1} < \frac{1}{2}$, and by definition of $a$, (\ref{eq:odd-sine-odd-l}) is equivalent to
\begin{equation*}
    a-2w \le j \le a\,.
\end{equation*}
Note that there are $2w+1$ such $j$'s, and so their sign contribution is -1. Therefore,
\begin{equation*}
\begin{aligned}
&
\prod_{j = 1}^{a}
\sgn\left(
\sin\left(
\frac{(4l+1)(2j-1)\pi}{8a}
\right)\right)\\
=&
\prod_{q = 1}^{(l-1)/2} \left(\prod_{j = (8w+2)(2q-1)+1}^{16wq+4q}
(-1)\right)
\cdot
\prod_{j = a-2w}^{a}
(-1)
=-1\,.
\end{aligned}
\end{equation*}
This completes Case 2, and we have
$$
\prod_{j = 1}^a \sgn\left(
\sin\left(
\frac{(4l+1)(2j-1)\pi}{8a}
\right)\right) = (-1)^l\,.
$$
for any positive integer  $l$\,.

Combining with (\ref{eq:odd-sine-decomp}), we obtain
\begin{equation*}
\sgn
\left
(\s_{k}
\left(
\prod_{j = 1}^{a} \sin
\left
(\frac{(2j-1)\pi}{8a}
\right)
\right)
\right)
= (-1)^{y+l}\end{equation*}
as claimed. This completes the proof of Lemma \ref{lem:odd-sine-2-8}.
\end{proof}

Now we are ready to prove Proposition \ref{prop:sign-D-2-8}. By definition and the quadratic reciprocity of Jacobi symbols, we have
\begin{equation}\label{eq:jac-k-a}
\jacobi{k}{a} = \jacobi{4l+1}{a} = \jacobi{a}{4l+1} = \jacobi{2l+1}{4l+1} = \jacobi{4l+1}{2l+1} = \jacobi{-1}{2l+1} = (-1)^{l}\,.
\end{equation}
Therefore, by Lemma \ref{l:jac}, 
$$
\s_k\left(\sqrt{a^a}\right) =\s_k\left(a^{\frac{a-1}{2}} \sqrt{a} \right) = \jacobi{k}{a}a^{\frac{a-1}{2}} \sqrt{a}  = (-1)^{l} a^{\frac{a-1}{2}} \sqrt{a}\,,
$$
i.e., $\sgn(\s_k(\sqrt{a^{a}})) = (-1)^l$. The proposition follows directly from Lemmas \ref{lem:even-sine-2-8}, \ref{lem:odd-sine-2-8} and  (\ref{eq:Dr}).
\end{proof}


\subsection{The higher central charge homomorphism $\Xi$ on $\WW_{\un}$} 

In this subsection, we construct certain infinite sequences $\ba_l=\{a_{l, n}\}_{n=0}^\infty$ of positive integers and prove that the restriction of the signature homomorphism $I$ on the subgroup $G_l$ of $\WW_\un$ generated by $\{[\CC_{a_{l,n}}]\mid n \ge 0\}$ has kernel $\langle [\II]\rangle$. Moreover, we show that $\Xi|_{G_l}$ is injective, and the image $\{\smap([\CC_{a_{l,n}}])\mid n \ge 0\}$ in the super-Witt group $\sW$ is  $\BF_2$-linearly independent.

For any positive integer  $l \equiv 2 \pmod{4}$, we define the sequence $\ba_l=\{a_{l, n}\}_{n= 0}^{\infty}$ inductively by setting   $a_{l, 0} = 2l+1$, and defining $a_{l, n+1}$ to be the smallest positive integer such that 
$$
a_{l, n+1} \equiv 2l+1 \pmod{8l+2}
$$
and $\gcd(a_{l, n+1}, a_{l, j}) = 1$ for all $j = 0, \ldots, n$. The existence of the infinite sequence $\ba_l$ is guaranteed by the Dirichlet prime number theorem. For example, the sequence $\ba_2$ begins with 
$$
5, 23, 41, 59, 77, 113, 131, 149, 167, 221, 239, \dots
$$
Let $G_{l, n}$ be the subgroup of $\WW_\un$ generated by $\{[\CC_{a_{l,j}}]\mid j = 0, \ldots, n\}$, and 
$$G_l = \bigcup_{n = 0}^{\infty} G_{l,n}\,.$$ 
For any positive integer $r$ and $\s\in\GQ$, we simply denote
$$\ss{r}(\s) := \ss{\CC_r}(\s) = \ss{\CC_r}'(\s)\,.$$

\begin{thm}\label{thm:signature}
Let $l\equiv 2 \pmod{4}$ be a positive integer. Then the set of signatures 
$
\{\ss{a_{l,n}}\}_{n \ge 0}
$
is $\BF_2$-linearly independent, i.e., if 
$\ss{a_{l,0}}^{b_0} \cdots \ss{a_{l,n}}^{b_n} = 1$ for some integers  $b_0, \ldots, b_n$ and positive integer $n$,  then  $b_0 \equiv \cdots \equiv b_n \equiv 0 \pmod{2}$. In particular, 
$$
I(G_l) = \bigoplus_{n \ge 0}\, \langle \ss{a_{l,n}}\rangle\,.
$$
Moreover, $ G_{l, n} \cap \ker I  = \langle [\II] \rangle$ for all nonnegative integer $n$.
\end{thm}
\begin{proof} Suppose $\{\ss{a_{l,n}}\}_{n \ge 0}$ is dependent. Then, there exist a positive integer $n$ and some integers  $b_0, \cdots, b_n$ such that $b_m$ is odd for some $m \le n$ and   $\ss{a_{l,0}}^{b_0} \cdots \ss{a_{l,n}}^{b_n} = 1$. 
Since $\ss{r}^2 = 1$ for all $r \in \BN$, we may assume that $b_j > 0$ by adding some positive multiple of $2$ if necessary. 
Now, let $\AA =\CC_{a_{l,0}}^{\bp{b_0}}\boxtimes \cdots \boxtimes \CC_{a_{l.n}}^{\bp{b_n}}$ and $A=[\AA]$.
Then,
\begin{equation} \label{eq:sig_1}
    \ss{\AA} = \ss{a_{l,0}}^{b_0} \cdots \ss{a_{l,n}}^{b_n} = 1 \,.
\end{equation}

Now let $N := \lcm\{\ord(T_{a_{l,j}})\mid 0 \leq j \leq n \text{ and } j \ne m \}$. Then by Lemma \ref{lem:ord-T} (a),
$$
32\mid N, \quad \mathrm{and}\quad 
N\mid 32\prod_{\substack{0\leq j\leq n\\ j \ne m}} a_{l,j}\,.
$$
By the definition of $\ba_l$, the integer $a_{l,m}$ is coprime to 32 and $a_{l,j}$ for all $j \ne m$, and hence $\gcd(a_{l,m}, N) = 1$. Therefore, there exist $x, y\in\BZ$ such that 
$$x a_{l,m} + yN = 1\,.$$ 
Set $k := -4lx a_{l,m} + 4l+1 =4lyN + 1$. Then $k \equiv 1 \pmod N$. Moreover, by Proposition \ref{prop:sign-D-2-8}, $\gcd(k,a_{l,m}) = 1$. Therefore, $\gcd(k, Na_{l,m})=1$. 

Let $\s \in \GQ$ such that $\s|_{\BQ_{Na_{l,m}}}=\s_k$. For any $j = 0, \ldots, n$ and $j \ne m$, we have $D_{a_{l,j}} \in \BQ_{\ord(T_{a_{l,j}})} \subset \BQ_{N}$ by Proposition \ref{prop:Dr} or \cite{dong2015}. Since $k \equiv 1 \pmod N$,  $\s(D_{a_{l,j}}) = D_{a_{l,j}}$ and hence $\ss{a_{l,j}}(\s) = 1$. Also, by Proposition \ref{prop:sign-D-2-8}, we have $$\ss{a_{l,m}}(\s) = (-1)^{lx/2}=(-1)^x$$
since $l/2$ is odd.
Note that $xa_{l,m} + yN = 1$ implies that $x$ has to be odd, and so we have
$$
\ss{\AA}(\s)=\prod_{j = 0}^{n}
\ss{a_{l,j}}(\s)^{b_j} =\ss{a_{l,m}}(\s)
=
(-1)^{x} = -1\,,
$$
which contradicts \eqref{eq:sig_1}. Therefore $\{\ss{a_{l,n}}\}_{n \ge 0}$
is $\BF_2$-linearly independent.

Since $I([C_{a_{l,0}}]^2) = \ss{a_{l,0}}^2= 1$ and  $\langle [C_{a_{l,0}}]^2 \rangle = \langle [\II] \rangle$, $\langle [\II] \rangle \subset G_{l,n}\cap \ker I$ for any nonnegative integer $n$.  Conversely, if $A \in G_{l,n}\cap \ker I$, then 
$$
A = [\CC_{a_{l,0}}]^{b_0} \cdots [\CC_{a_{l.n}}]^{b_n} \quad \text{and}\quad I(A)= \ss{a_{l,0}}^{b_0} \cdots \ss{a_{l,n}}^{b_n} = 1\,.
$$
The preceding conclusion implies $b_0, \dots, b_n$ are all even. Since $[\CC_r]^2 \in \langle [\II] \rangle$ for all $r$, we find $A \in \langle [\II] \rangle$. 
Therefore, $ G_{l, n} \cap \ker I  = \langle [\II] \rangle$.
\end{proof}

The subgroup generated by $\{\smap([\CC_r]) \mid r\ge 1\}$ in $\sW$ is an abelian group of exponent 2. It is conjectured in \cite[Conjecture 5.21]{DNO} that $\{\smap([\CC_r]) \mid r\ge 1\}$ is linearly independent in $\sW$. The following corollary proves that this holds for infinitely many subsequences of $\{\smap([\CC_r]) \mid r\ge 1\}$, but we do not know whether they are in $\sW_2$ or not. The group $\sW_2$ will be further discussed in Section \ref{s:sW}.

\begin{cor}\label{c:indep}
For any positive integer $l \equiv 2 \pmod{4}$, the sequence
$$
\bsa_l = \{\smap([\CC_{a_{l,j}}])\mid j \ge 0\}
$$
is linearly independent in $\sW$.
\end{cor}
\begin{proof}
Apply the signature homomorphism $sI$ to the sequence $\{\smap([\CC_{a_{l,n}}]) \mid n\ge 0\}$.  By Corollary \ref{c:comm}, we find
$$
\{sI\circ \smap([\CC_{a_{l,n}}])\}_{n\ge 0} = \{\ss{s_{l,n}}\}_{n\ge 0} 
$$
which is $\BF_2$-linearly independent by Theorem \ref{thm:signature}. Therefore, $\{\smap([\CC_r]) \mid r\ge 1\}$ is linearly independent in $\sW$.
\end{proof}

The commutativity of the  diagram
$$
\begin{tikzcd}
G_l \ar[rr, "\smap"] \ar[rd, "I"']
&&
\smap(G_l) \ar[ld, "sI",  "\sim"{sloped}]\\
&
I(G_l) 
&
\end{tikzcd}
$$
implies the equivalence of the restriction  of  $\smap$ and $I$ on $G_l$. Now, we can determine the isomorphism class of $G_{l}$ using the higher central charges homomorphism $\Xi$ or simply the signature homomorphism $I$.

\begin{cor}\label{c:Gz}
For any positive integer $l \equiv 2 \pmod{4}$,  $G_l$ has the direct sum decomposition:
\begin{equation} \label{eq:decomp}
    G_l = \langle [\CC_{a_{l,0}}] \rangle \oplus  \bigoplus_{n \ge 1} \langle C_n \rangle   \cong \BZ/{32} \oplus (\BZ/2)^{\oplus \BN}\,.
\end{equation}
where $C_n = [\CC_{a_{l, n}}] \cdot [\CC_{a_{l, 0}}]^{-i_n}$ for $n \ge 1$, where $i_n$ is an integer such that $[\CC_{a_{l, 0}}]^{2 i_n} =  [\CC_{a_{l, n}}]^2$.   Moreover, $\Xi|_{G_l}$ is injective.  
\end{cor}
\begin{proof} 
For any  integer $n \ge 0$, since $\langle [\CC_{a_{l, n}}]^2 \rangle = \langle [\CC_{a_{l, 0}}]^2 \rangle$, there exists an integer $i_n$ such that  $[\CC_{a_{l, n}}]^2 = [\CC_{a_{l, 0}}]^{2 i_n}$. Then $C_n^2 = [\Vs{}]$ for $n \ge 1$, and $G_l$ is generated  by the elements  $[\CC_{a_{l,0}}]$ and $ C_n$, $n \ge 1$. Suppose 
$$
[\Vs{}]=[\CC_{a_{l,0}}]^{b_0} C_1^{b_1} \cdots C_n^{b_n}
$$
for some positive integer $n$ and integers $b_0, \dots, b_n$. Note that $\ss{C_j} = \ss{a_{l,j}}$, and so we have
$$
1 =  \ss{a_{l,0}}^{b_0-(i_1 b_1 +\cdots +i_n b_n)} \ss{a_{l,1}}^{b_1} \cdots \ss{a_{l,n}}^{b_n} \,.
$$
By Theorem \ref{thm:signature}, we have $b_0, b_1, \dots, b_n$ are all even and hence $C_j^{b_j}=[\Vs{}]$ for $j >0$. Thus, we have $[\CC_{a_{l,0}}]^{b_0}=[\Vs{}]$, and this proves the direct sum decomposition \eqref{eq:decomp}. 
The second isomorphism follows immediately from the fact that $\ord([\CC_{a_{l,0}}])=32$. 

Finally, if $A = [\CC_{a_{l,0}}]^{b_0} C_1^{b_1} \cdots C_n^{b_n} \in \ker(\Xi)$, then $I(A) =1$ and $\xi_1(A)=1$ by Proposition \ref{p:ker}. It follows from Theorem \ref{thm:signature} that $b_0, \dots, b_n$ are all even and so $A = [\CC_{a_{l,0}}]^{b_0} \in \langle [\II]\rangle$. Since the first central charge homomorphism $\xi_1$ is injective on $\langle [\II]\rangle$, $\xi_1(A) = 1$ implies $A=[\Vs{}]$. Therefore, $\Xi$ is injective on $G_l$.
\end{proof}

\begin{remark}\label{r:subseq}
By the same argument, Corollary \ref{c:Gz} holds for any infinite subsequence of $\ba_l$. This version of Corollary \ref{c:Gz} will be used in Theorems \ref{t:sqroot} and \ref{thm:sW2}.
\end{remark}

It is not difficult to show that the higher central charge homomorphism $\Xi$ is injective on $\WW_{\pt}(p)$, the Witt subgroup of $\WW_\pt(p)$ generated by the Witt classes of the pointed modular categories $\CC(H,q)$  where $H$ is a finite abelian $p$-group. However, the kernel of $\Xi|_{\WW_{\pt}}$ is not trivial. 

For any odd prime $p$, let $A_p$ be the unique Witt class in $\WW_\pt(p)$ of trivial signature. In particular, $A_p$ has order 2 and can be represented by an abelian group of order $p^2$. Let $A_2 := [\II]^8$,  which is the unique class of $\WW_{\pt}(2)$ of order 2 with trivial signature. Therefore, $A_p$ is the unique element of $\WW_\pt(p)$ of order 2 with trivial signature for all prime $p$.

\begin{remark}\label{r:Ap}
Note that the first central charge of $A_p$ is $-1$ for any prime $p$, and so $A_p \not\in \ker(\Xi)$. However, for any primes $p$ and $p'$, $A_p A_{p'} \in \ker (\Xi)$ by Proposition \ref{p:ker}.
\end{remark}

\begin{prop}\label{prop:ker-Xi-pt}
The intersection $\ker(\Xi)\cap\WW_\pt$ is generated by $A_{p}A_{p'}$, where $p$ and $p'$ are distinct primes.
\end{prop}
\begin{proof}
Let $B \in \ker(\Xi)\cap\WW_\pt$ be a non-trivial Witt class. Then there exists an anisotropic metric group $(H, q)$ such that $B = [\CC(H, q)]$ such that $\xi_1(\CC(H, q)) = 1$, and  $I(B)=1$ by Proposition \ref{p:ker}.

Let $h = \prod_{j = 1}^{n} p_{j}^{e_{j}}$ be the order of $H$, where $p_1, \dots, p_n$ are distinct primes. Then 
$$
B = B_{p_1}\cdots B_{p_n}
$$
where $B_{p_j} \in \WW_\pt(p_j)$ is nontrivial for $j=1, \dots, n$. 

We claim that  $e_j$ has to be even for all $j = 1, ..., n$. In particular, $I(B_{p_j})=1$. Otherwise, $\sqrt{h} \not\in \BQ$ and there exists $\s \in \GQ$ such that $\s(\sqrt{h})=-\sqrt{h}$, that means $I(B)(\s)=-1$, 
a  contradiction.

To complete the proof, it suffices to show $n$ is even and $B_{p_j} = A_{p_j}$ for $j=1, \dots, n$.
Since $e_j$ is even for any $j = 1, ..., n$,  $B_{p_j} = A_{p_j}$ if $p_j$ is odd. Thus, if $h$ is odd, so are $p_j$ for all $j$. Then 
$$
\prod_{j = 1}^{n} A_{p_j}  = \prod_{j = 1}^{n} B_{p_j} = B \in \ker(\Xi)
$$
implies that $n$ must be even by Remark \ref{r:Ap}.

We now consider the case when $h$ is even. Then  $h = 2^{e_1}\prod_{j = 2}^{n}p_j^{e_{j}}$ for some even positive integer $e_1$. Suppose $n$ is odd.  Then
$$
\prod_{j = 2}^{n} B_{p_{j}} =\prod_{j = 2}^{n} A_{p_{j}} \in \ker(\Xi)
$$
by Remark \ref{r:Ap}.  Since $B \in \ker(\Xi)$, and so $B_2\in \ker(\Xi)$. However, this contradicts that $\Xi$ is injective on $\WW_\pt(2)$ (cf. \cite[Example 6.2]{NSW}). Therefore, $n$ is also even in this case. Consequently, $B_2$ has the same first central charge as $\prod_{j = 2}^{n} A\inv_{p_{j}}$, which is $-1$. Therefore, $B_2=A_2$ as $I(B_2)=1$.
\end{proof}

Theorem \ref{thm:signature} and Proposition \ref{prop:ker-Xi-pt} inspire the following questions.

\begin{question}
Is the sequence $\{\ss{r}\}_{r\ge 0}$ linearly independent? An affirmative answer to this question implies the linearly independence of the sequence $\{\smap([\CC_r])\}_{r\ge 0}$ in $\sW$. 
\end{question}

\begin{question}
Is the intersection between $\ker(\Xi)$ and the torsion subgroup $\Tor(\WW_{\un})$ contained in $\WW_\pt$?
\end{question}

\section{The group $s\WW_2$} \label{s:sW}
Let $\sW_{\pt}$ be the image $\smap(\WW_{\pt})$ in $\sW$. By \cite[Proposition 5.18]{DNO}, the super-Witt group $\sW$ can be decomposed into a direct sum
$$
\sW = \sW_{\pt} \oplus \sW_2 \oplus \sW_\infty
$$
where $\sW_2$ (resp.~$\sW_\infty$) is the subgroup of $\sW$ generated by the Witt classes of completely anisotropic s-simple fusion categories of Witt order 2 (resp.~of infinite Witt order). In particular, the torsion part of $\sW$ is $\sW_{\pt} \oplus \sW_2$. It is conjectured \cite[Conjecture 5.21]{DNO} that $\sW_2$ has infinite rank.  In this section, we give a proof for this conjecture in Theorem \ref{thm:sW2}. We also prove that $[\II]$ has infinitely many square root in $\WW$ modulo $\WW_\pt$ (Theorem \ref{t:sqroot}).

Fix a positive integer $l \equiv 2 \pmod{4}$ as before. Consider the subsequence $\bq=\{p_j\}_{j=0}^\infty$ of $\ba_l$ consisting of all prime number terms.  Again by the Dirichlet prime number theorem, $\bq$ is an infinite sequence. Let $G_\bq$ be the subgroup of $G_l$ generated by $\{[\CC_{p_j}]\}_{j \ge 0}$.

\begin{prop}\label{prop:G-and-pt}
We have ${G_\bq} \cap \WW_\pt = \langle [\II]^2\rangle$.
\end{prop}
\begin{proof}
Since $[\II]^2 \in \WW_\pt$ and $\langle [\II]^2 \rangle = \langle [\CC_{p_0}]^4 \rangle$, $\langle [\II]^2 \rangle \subseteq {G_\bq} \cap \WW_\pt$. 
Suppose  $A \in G_\bq \cap \WW_\pt$  is nontrivial. Then $A = \prod_{j = 0}^{n}\witt{\CC_{p_{j}}}^{b_{j}}$ for some integers $b_0,\dots, b_n$. As is illustrated in the proof of Theorem \ref{thm:signature}, we can assume that all the $b_j$'s are nonnegative, and we let $\AA=\CC_{p_{0}}^{\bp{b_{0}}}\boxtimes \cdots \boxtimes \CC_{p_{n}}^{\bp{b_{n}}}$. 

We first show that all the $b_j$'s are even. If not, there exists a nonnegative integer $m \le n$ such that $b_m$ is odd. Since $[\CC_{p_j}]^2 \in \langle [\CC_{p_0}]^2\rangle$ for all $j$, we may simply assume $m=n$.
 
Since $A \in \WW_\pt$, there exists a finite abelian group $H$ and a non-degenerate quadratic form $q: H \to \BC^\times$ such that the corresponding pseudounitary modular category $\HH=\CC(H, q)$  satisfies $[\HH] = A$ in $\WW_{\un}$. By Theorem \ref{thm:cat-dim-sgn},  
\begin{equation}\label{eq:same-sgn}
\ss{\AA}(\s) = \ss{\HH}(\s)
\end{equation}
for all $\s \in \GQ$.

Let $h$ denote the order of  $H$. Since $p_{n}$ is an odd prime different from $p_0, \dots, p_{n-1}$, we can write $h$ as $h = h_1h_2$, where $\gcd(h_1, p_{n}) = 1$, and $h_2=p_{n}^s$ for some $s \in \BN\cup\{0\}$. Let
$$
M:= 32\cdot h_1 \cdot \prod_{j = 0}^{n-1} p_{j}\,.
$$
Again, Lemma \ref{lem:ord-T} (a) implies that  $\ord(T_{p_{j}})\mid M$ for $j=0,\dots, n-1$. Moreover, by construction, $\gcd(p_{n}, M) = 1$. Hence, there exist $x, v\in \BZ$ such that 
$$
xp_{n} + vM = 1\,.
$$
Note that since $M$ is even, $x$ has to be odd. 

Set $k:=-4lxp_{n} + 4l+1 = 4lvM+1$. Then $k \equiv 1\pmod{M}$ and $\gcd(k, p_{n}) = 1$ by Proposition 6.3. Therefore, we have $k \equiv 1\pmod{h_1}$, and $\gcd(k, p_{n}M) = 1$.

Let $N=Mp_{n}$ and $\s\in \GQ$ such that $\s|_{\BQ_{N}} = \s_k$. By Proposition 5.2 or \cite{dong2015} again,  $D_{p_{j}} \in \BQ_{\ord(T_{p_{j}})}\subset \BQ_M$ for $j = 0, \ldots, n-1$. Since $k\equiv 1\pmod{M}$,  $\s(D_{p_{j}}) = D_{p_{j}}$, and hence $\ss{p_{ j}}(\s) = 1$ for $j = 0, .., n-1$. Apply Proposition \ref{prop:sign-D-2-8}, we have
$$
\ss{p_{n}}(\s) = (-1)^{lx/2} = -1\,.
$$
Since $b_n$ is odd, we have
$$
\ss{\AA}(\s) = \prod_{j = 0}^{n}\ss{p_{j}}(\s)^{b_j} = \ss{p_{n}}(\s)^{b_n}= -1\,.
$$

By the definition of $M$, we have $\sqrt{h} \in\BQ_{N}$ and $\sqrt{h_1} \in \BQ_{M}$. On one hand, since $k \equiv 1\pmod{M}$, $\s(\sqrt{h_1}) = \sqrt{h_1}$. On the other hand, $k \equiv 1\pmod{4}$. By the same computation as \eqref{eq:jac-k-a}, we have
$$
\jacobi{k}{p_{n}} = (-1)^l  = 1\,.
$$
Therefore, by Lemma \ref{l:jac} and $h_2 = p_{n}^s$, we have
\begin{equation}
    \ss{\HH}(\s) = \frac{\s(\sqrt{h_1})}{\sqrt{h_1}}\frac{\s(\sqrt{h_2})}{\sqrt{h_2}} = \frac{\s(\sqrt{h_2})}{\sqrt{h_2}} = \jacobi{k}{h_2} = \jacobi{k}{p_{n}}^s = 1\,,
\end{equation}
contradicting \eqref{eq:same-sgn}. Therefore, $b_0, \dots, b_n$ are all even, and hence $A \in \langle [\II] \rangle$. 

Since $\langle [\II] \rangle \cap \WW_\pt = \langle [\II]^2 \rangle$, $A \in \langle [\II]^2 \rangle$. Therefore, $G_\bq \cap \WW_\pt = \langle [\II]^2 \rangle$.
\end{proof}

\begin{thm} \label{t:sqroot}
The Witt class of the Ising modular category $\II$ has infinitely many square roots in $\WW$ modulo $\WW_\pt$.
\end{thm}
\begin{proof}
Let $\bq=\{p_j\}_{j=0}^\infty$  be the prime number subsequence of $\ba_l$. Then $G_\bq \cap \WW_\pt =\langle[\II]^2\rangle$. By Corollary \ref{c:Gz} and Remark \ref{r:subseq}, 
$$
G_{\bq} =   [\CC_{p_0}]\oplus \bigoplus_{n > 0} \langle C_n\rangle
$$
where $C_n$ is an order 2 element given by $[\CC_{p_n}][\CC_{p_0}]^{-i_n}$ for some integer $i_n$.
Since $\langle [\CC_{p_n}]^2\rangle= \langle[\II]\rangle$ for all $n \ge 0$, the statement follows.
\end{proof}

\begin{thm}\label{thm:sW2} The group $s\WW_2$ has infinite rank.
\end{thm}
\begin{proof}
Let $\bq=\{p_j\}_{j=0}^\infty$  be the prime number subsequence of $\ba_l$.  Then $\smap(G_\bq)$
is an elementary 2-group, and so $\smap(G_\bq +\WW_{\pt}) \subset s\WW_{\pt}\oplus s\WW_2$ by \cite[Proposition 5.18]{DNO}. By \cite[Proposition 5.18]{DNO},   $\sW_{\pt} = \smap(\WW_{\pt})$. Thus,  
$$
 \frac{\smap(G_\bq+\WW_{\pt})}{\smap(\WW_{\pt})} 
$$
is isomorphic to a subgroup of 
$ \sW_2$.
By Proposition \ref{prop:G-and-pt}, we have $G_\bq \cap \WW_\pt = \langle[\II]^2\rangle$. Since $\ker(\smap)= \langle[\II]\rangle$, by Corollary \ref{c:Gz} and Remark \ref{r:subseq}, we have
$$
\begin{aligned}
\frac{\smap(G_\bq +\WW_\pt)}{\smap(\WW_\pt)} \cong \frac{G_\bq+\WW_\pt}{\langle[\II]\rangle + \WW_\pt} &\cong 
\frac{G_\bq+\WW_\pt}{\WW_\pt}\bigg/\frac{\langle[\II]\rangle + \WW_\pt}{\WW_\pt}\\ 
&\cong
\frac{G_\bq}{G_\bq \cap \WW_\pt}\bigg/\frac{\langle[\II]\rangle }{G_\bq\cap  \WW_\pt}
\cong
\frac{G_\bq}{\langle[\II]\rangle} \cong (\BZ/2)^{\oplus \BN}\,.
\end{aligned}
$$
Therefore, $\sW_2$, is an infinite group.
\end{proof}

\section*{Acknowledgements}
The authors would like to thank Zhenghan Wang for his discussions on the square roots of Ising modular categories. They would also like to thank the referee for their invaluable suggestions to improving this paper. The first and the third authors would like to thank Ling Long for many fruitful discussions on the signatures of real algebraic integers.  This paper is based upon work supported by the National Science Foundation under the Grant No. DMS-1440140 while the first two and the last authors were in residence at the Mathematical Sciences Research Institute in Berkeley, California, during the Spring 2020 semester. 

\bibliographystyle{abbrv}
\bibliography{ref}

\end{document}